\documentclass[10pt]{amsart}
\usepackage{amsmath, amssymb, amsopn}
\usepackage{array}
\usepackage{amscd}
\usepackage{epsfig}
\usepackage{color}
\usepackage{comment}
\usepackage{setspace}
\usepackage{textcomp}
\usepackage{stmaryrd}
\usepackage{multirow}
\usepackage{diagbox}
\usepackage{todonotes}
\usepackage{subcaption}
\usepackage{placeins}
\usepackage{xcolor}
\usepackage{mathtools}
\usepackage{multirow}
\usepackage{xcolor}
\usepackage{longtable}
\usepackage{booktabs}
\usepackage{url}

\newtheorem{theorem}{Theorem}[section]
 
\newtheorem{lemma}[theorem]{Lemma}

\newtheorem{definition}[theorem]{Definition}

\newcommand{\pmat}[1]{\begin{pmatrix} #1 \end{pmatrix}}
\newcommand{\case}[1]{\begin{cases} #1 \end{cases}}

\newcommand{\algn}[1]{\begin{align} #1 \end{align}}
\newcommand{\algns}[1]{\begin{align*} #1 \end{align*}}
\newcommand{\mltln}[1]{\begin{multline} #1 \end{multline}}

\newcommand{\gat}[1]{\begin{gather} #1 \end{gather}}

\newcommand{\barr}{\begin{array}}
\newcommand{\earr}{\end{array}}

\newcommand{\half} {\ensuremath{\frac{1}{2}}}
\newcommand{\mc}[1]{\mathcal{#1}}
\newcommand{\LRp}[1]{\left( #1 \right)}

\newcommand{\LRa}[1]{\left< #1 \right>}

\newcommand{\LRc}[1]{\left\{ #1 \right\}}
\newcommand{\jump}[1] {\ensuremath{[\![{#1}]\!]}}

\newcommand{\avg}[1] {\ensuremath{\left\{\!\!\left\{{#1}\right\}\!\!\right\}}}

\newcommand{\ra}{\rightarrow}

\newcommand{\R}{{\mathbb R}}

\newcommand{\M}{{\mathbb M}}

\newcommand{\calT}{\mathcal{T}}

\renewcommand{\div}{\operatorname{div}}

\newcommand{\osc}{\operatorname{osc}}

\newcommand{\bs}{\boldsymbol}

\newcommand{\pd}{\partial}
\newcommand{\nor}[1]{\left\Vert#1\right\Vert}
\newcommand{\norw}[2]{\left\Vert#1\right\Vert_{#2}}

\newcommand{\n}{\bs{n}}

\newcommand{\tred}[1]{\textcolor{black}{#1}}

\newcommand{\Cc}{C_{\text{\rm{co}}}}

\newcommand{\Div}{\nabla \cdot}

\newcommand{\Ehbdy}{\mc{E}_h^{\pd}}
\newcommand{\EhD}{\mc{E}_h^{D}}

\newcommand{\Ehi}{\mc{E}_h^{0}}
\newcommand{\Eh}{\mc{E}_h}

\newcommand{\F}{\mc{F}}

\newcommand{\gammakap}{\gamma}

\newcommand{\inte}{\operatorname{int}}
\newcommand{\kap}{\ul{\bs{\kappa}}}

\newcommand{\Pb}{\bs{\mc{P}}}
\newcommand{\Pic}{\Pi_h^c }
\newcommand{\Pin}{\Pi_h^0 }

\newcommand{\qb}{\bs{q}}

\newcommand{\tnorm}[1]{{\left\vert\kern-0.25ex\left\vert\kern-0.25ex\left\vert#1\right\vert\kern-0.25ex\right\vert\kern-0.25ex\right\vert}}
\newcommand{\Th}{\mc{T}_h}

\newcommand{\uh}{u_h}
\newcommand{\uhc}{u_h^c}
\newcommand{\uhn}{u_h^0}

\newcommand{\ul}{\underline}

\newcommand{\vc}{v^c}
\newcommand{\vn}{v^0}

\newcommand{\VBh}{\mathbb{V}_h}
\newcommand{\Vh}{V_h}
\newcommand{\Vhc}{V_h^c}
\newcommand{\Vhn}{V_h^0}
\newcommand{\Vhnn}{V_h^{00}}
\newcommand{\Vhperp}{\VBh^{\perp}}
\newcommand{\vtiln}{\tilde{v}^0}

\newcommand{\wc}{{w^c}}
\newcommand{\wn}{{w^0}}

\newcommand{\whc}{w_h^c}
\newcommand{\whn}{w_h^0}

\newcommand{\Xc}{X_{c,\pd}}
\newcommand{\Xn}{X_{c,0}}
\newcommand{\Yc}{X_{0,\pd}}
\newcommand{\Yn}{X_{0,0}}

\newcommand{\zb}{\bs{z}}

\begin{document}

\title[Interior over-penalized enriched Galerkin methods]{Interior over-penalized enriched Galerkin methods for second order elliptic equations }

\author{Jeonghun J. Lee} 
\address{Department of Mathematics, Baylor University, Waco, TX 76706, USA}
\email{jeonghun\_lee@baylor.edu}
\urladdr{}
\author{Omar Ghattas}
\address{Oden Institute for Computational Engineering and Sciences, University of Texas at Austin, Austin, TX 78712, USA}
\email{omar@oden.utexas.edu}
\subjclass[2000]{Primary: 65N30}
\keywords{enriched Galerkin methods, error analysis, preconditioning}

%
%

\maketitle
\begin{abstract}
In this paper we propose a variant of enriched Galerkin methods for second order elliptic equations with \tred{over-penalization} of interior jump terms. 
The bilinear form with interior \tred{over-penalization} gives a non-standard norm which is different from the discrete energy norm in the classical discontinuous Galerkin methods. Nonetheless we prove that optimal a priori error estimates with the standard discrete energy norm can be obtained by combining a priori and a posteriori error analysis techniques. We also show that the interior \tred{over-penalization} is advantageous for constructing preconditioners robust to mesh refinement by analyzing spectral equivalence of bilinear forms. Numerical results are included to illustrate the convergence and preconditioning results.
\end{abstract}

\section{Introduction}

It is well known that numerical fluxes of Galerkin methods with Lagrange finite elements (CG) do not satisfy local mass conservation.
Since local mass conservation is an important physical principle in numerical simulations, 
numerical methods providing locally mass conservative numerical flux with/without postprocessing
have been intensively studied. These include post-processing of the CG methods \cite{Cockburn-Gopalakrishnan-Wang-2007,Chippada-Dawson-Martinez-Wheeler,Hughes-Engel-Mazzei-Larson}, the mixed method \cite{BDM85,RT75,Nedelec80,Nedelec86}, and a class of discontinuous Galerkin (DG) methods \cite{Arnold82,ABCM02,Wheeler73, HDG-first,Wang-Ye-weak-galerkin}), to name a few.

Enriched Galerkin (EG) methods \cite{Sun-Liu-2009} are proposed to achieve locally mass conservative numerical methods with low computational costs. 
Compared to classical DG methods for the primal formulation, EG methods have \tred{much} fewer degrees of freedom and provide a simple local post-processing for reconstruction of locally mass conservative numerical fluxes. For efficient solvers, block diagonal preconditioners based on algebraic multigrid methods have been proposed as efficient preconditioners for EG methods \cite{Lee-Lee-Wheeler-2016}. However, the preconditioning analysis in this work relies on the full elliptic regularity assumption which requires restrictive conditions on the domain geometry and material coefficients. In fact, numerical experiments show that the block preconditioners in \cite{Lee-Lee-Wheeler-2016} are not scalable with mesh refinement when the boundary condition is not a pure Neumann boundary condition.

The purpose of this paper is to develop new EG methods that retain the advantages of local mass conservation \tred{while} also providing more robust fast solvers that do not need \tred{restrictive} assumptions on domain geometry and boundary conditions. The key in the development of these new EG methods is to replace the original bilinear form with a new bilinear form that has an {\color{black}{over-penalized}} interior jump term. As such, we call them interior {\color{black}{over-penalized}} enriched Galerkin methods ({\color{black}{IOP-EG}}). The \tred{over-penalization} parameters depend on the mesh sizes, so the {\color{black}{over-penalized}} jump term and other terms in the new bilinear form do not have the same scaling property. Because of these incompatible scaling properties, the standard techniques for error analysis are not available for \tred{the} optimal error estimates \tred{of} the new EG methods. 

The first main result in this paper is an error analysis with optimal error estimates. 
To overcome the difficulty \tred{caused by the bilinear forms with} different \tred{scalings}, we use a medius analysis idea utilizing a posteriori error estimate results for our a priori error estimates. 
\tred{Through this idea} we can obtain optimal error estimates for {\color{black}{IOP-EG}} methods. 
The second main result is \tred{the} construction of scalable preconditioners for the {\color{black}{IOP-EG}} methods.
More precisely, we propose an abstract form of preconditioners based on appropriate mesh-dependent norms and carry out spectral equivalence analysis for the abstract preconditioners following the framework in \cite{Mardal-Winther-2011}. It turns out that \tred{over-penalization} is crucial for construction of scalable preconditioners via an analysis of spectral equivalence.  

The paper is organized as follows. In Section \ref{sec:prelim} we introduce notation and define the {\color{black}{IOP-EG}} methods. 
In Section \ref{sec:apriori}, we prove optimal error estimates of the {\color{black}{IOP-EG}} methods with a minimal regularity assumption using a medius error analysis. 
Finally, numerical results illustrating our theoretical results and concluding remarks will be given in Sections \ref{sec:numerical} and \ref{sec:conclusion}.

\section{Preliminaries} \label{sec:prelim}
\subsection{Notations}
For a set $D \subset \R^n$ we use $L^2(D)$ to denote the space of square integrable functions on $D$. For a finite-dimensional vector space $\Bbb{X}$, $L^2(D; \Bbb{X})$ is the space of $\Bbb{X}$-valued square-integrable functions on $D$ whose inner product is naturally defined with the inner product on $\Bbb{X}$ and the inner product on $L^2(D)$. 

Let $\Omega$ be a bounded polygonal/polyhedral domain in $\R^n$ with $n=2$ or $3$. 
For a nonnegative real number $s$, $H^s(\Omega)$ denotes the standard Sobolev spaces based on the $L^2$ norm (see \cite{Evans-book} for details). 
We use $\calT_h$ to denote a shape-regular triangulation of $\Omega$ where $h$ is the maximum diameter of triangles or tetrahedra in $\mc{T}_h$. 
We also use $\mathcal{E}_h$ to denote the set of facets in $\mc{T}_h$, i.e., the set of $(n-1)$-dimensional simplices in the triangulation determined by $\mc{T}_h$. In particular, 
$\Ehi$ and $\Ehbdy$  are the sets of interior and boundary facets of $\mc{T}_h$, respectively.

For $e \in \mathcal{E}_h$ and functions $f, g \in L^2(e)$ we define $\langle {f}, {g} \rangle_e = \int_e {f} \cdot {g} \,ds$. For a set $\Gamma$ which is a union of facets in $\Eh$, we define 
\begin{align*}
\langle f, g \rangle_{\Gamma} = \sum_{e \subset \Gamma} \langle {f}, {g} \rangle_e. 
\end{align*}
This notation is naturally extended to $\LRa{ \bs{f}, \bs{g} }_{\Gamma}$ for vector-valued functions $\bs{f}, \bs{g} \in L^2(e; \R^n)$ with the natural inner product on $\R^n$. 

For an integer $k \geq 0$ 
and for each $T \in \mc{T}_h$, $\mc{P}_k(T)$ is the space of polynomials of degree $\le k$ on $T$, and $\mc{P}_k(\mc{T}_h)$ denotes the space 
\algns{
\mc{P}_k(\mc{T}_h) = 
\case{ 
\LRc{q \in H^1(\Omega) \;:\; q|_T \in \mc{P}_k(T), \; T \in \mc{T}_h } \quad \text{if } k \ge 1 \\
\LRc{q \in L^2(\Omega) \;:\; q|_T \in \mc{P}_k(T), \; T \in \mc{T}_h } \quad \text{if } k = 0  \\
} .
}
We use $X \lesssim Y$ to denote an inequality $X \le CY$ with a constant $C$ which depends only on the shape regularity of $\mc{T}_h$ and polynomial degree $k$,
and $X \sim Y$ stands for $X \lesssim Y$ and $Y \lesssim X$.

\subsection{Enriched Galerkin methods with interior \tred{over-penalization}}
In this subsection we introduce the interior {\color{black}{over-penalized}} enriched Galerkin methods.

For $\Vhc = \mc{P}_k(\mc{T}_h)$ with $k\ge 1$ and $\Vhn = \mc{P}_0(\mc{T}_h)$, we define $\Vh$ as the linear space 
\algns{
\Vh = \{ \vc + \vn \,:\, \vc \in \Vhc, \vn \in \Vhn \}. 
}
A function $v \in \Vh$ with $v = \vc + \vn$ is determined by the standard degrees of freedom of $\Vhc$ and $\Vhn$. 
However, the degrees of freedom of $\Vh$ is not the union of the two sets of degrees of freedom, of $\Vhc$ and of $\Vhn$. In fact, a piecewise polynomial function $v \in \Vh$ may have more than one expression as a sum of elements in $\Vhc$ and $\Vhn$. 
If we use $v = \vc + \vn$ with $\vc \in \Vhc$, $\vn \in \Vhn$ to denote an element in $\Vh$, the zero function can have infinitely many expressions by choosing $\vc = -\vn = C$ for any constant $C$. However, such non-unique expressions are unique up to constant addition. To see it, assume that $v = \vc_1 + \vn_1 = \vc_2 + \vn_2$. Then $\vc_1 - \vc_2 = \vn_2 - \vn_1 \in \Vhc \cap \Vhn$, so $\vc_1$ and $\vc_2$ are the same up to constant addition. So are $\vn_1$ and $\vn_2$. 

We use $\VBh$ to denote the space $\Vhc \times \Vhn$. 
Note that one element in $\phi \in \VBh$ uniquely determines $\phi_c \in \Vhc$ and $\phi_0  \in \Vhn$. 
Let $\VBh^{ker}$ be the one dimensional subspace $\VBh^{ker}:= \{ \phi \in \VBh\,:\, \phi_c = - \phi_0 \}$ in $\VBh$.
If we denote the quotient space $\VBh / \VBh^{ker}$ by $\Vhperp$, there is a one-to-one correspondence between $\Vh$ and $\Vhperp$. 

To introduce the bilinear form for the enriched Galerkin method, we define the jump and average operators of functions which have well-defined traces on edges/faces in the context of classical discontinuous Galerkin methods. 
For $e \in \mc{E}_h^{\pd}$
\algns{
	\jump{q}|_e = q|_e \n , \quad \avg{q}|_e = q|_e, \qquad \jump{\qb}|_e = \qb|_e \cdot \n, \quad \avg{\qb}|_e = \qb|_e 
}
where $q$ and $\qb$ are $\R$- and $\R^n$-valued functions such that their traces on $e$ are well-defined, and $\n$ is the outward unit normal vector field on $e$. For $e \in \mc{E}_h^0$, let $T_+$ and $T_-$ be the two elements sharing $e$ as $e = \pd T_+ \cap \pd T_-$. If $q$ is a scalar function on $T_+ \cup T_-$, then we use $q^{\pm}$ and $\nabla q^{\pm}$ to denote the restrictions of $q$ and $\nabla q$ on $T_+$ and $T_-$. We use $\n^{\pm}$ to denote the unit outward normal vector fields of $T_\pm$ and we again omit the restriction $|_e$ if it is clear in context. Then the jumps and the averages of $q$ and $\qb$ on $e \in \Ehi$ are defined by 
\algns{
\jump{q}|_e &= q^+ \n^+ + q^- \n^-, & \avg{q}|_e &= \half (q^+ + q^- ), \\
\jump{\qb}|_e &= \qb^+ \cdot \n^+ + \qb^- \cdot \n^- , & \avg{\qb}|_e &= \half (\qb^+ + \qb^- ) .
}
If $e \in \mc{E}_h$ is clear in context, we will use $\jump{q}$ instead of $\jump{q}|_e$ for the jump of $q$ on $e \in \mc{E}_h$. The same simplification will apply to other quantities.

We consider a model second order elliptic equation 
\gat{
\label{eq:poisson-eq} - \Delta u = f \qquad  \text{ in } \Omega, \\
\label{eq:poisson-bc} u|_{\Gamma_D} = u_D \text{ on } \Gamma_D, \quad \nabla u \cdot \n |_{\Gamma_N} = u_N \text{ on } \Gamma_N 
}
{\color{black}{where $\Gamma_D$, $\Gamma_N$ are disjoint open subsets of $\pd \Omega$, the boundary of $\Omega$, such that $\overline{\Gamma_D} \cup \overline{\Gamma_N} = \pd \Omega$.}}

Let $\gamma >0$ be a sufficiently large positive constant. 
We use $\EhD$ to denote the set $\{ e \in \Ehbdy \,:\, e \subset \Gamma_D \}$. For {\color{black}{$V :=H^1(\Omega)$}},
%
let $a : V \times \Vh \ra \R$ be
\algn{ \label{eq:a}
a (v, w) &= \LRp{ \nabla v, \nabla w} + \LRa{ \gammakap h_e^{-1 } {v}, {w} }_{\EhD} 
}
where 
\algns{
\LRa{ \gammakap h_e^{-1 } {v}, {w} }_{\EhD}  := \sum_{e \in \EhD} \LRa{ \gammakap h_e^{-1} v, w}_e 
}
with $h_e$, the diameter of $e \in \Eh$. 
We also define the bilinear form $a_h$ on $\Vh \times \Vh$ as 
\algn{ \label{eq:ah}
a_h (v, w) &= \LRp{ \nabla v, \nabla w} - \LRp{ \LRa{ \avg{\nabla v}, \jump{w}}_{\Ehi\cup\EhD} + \LRa{ \jump{v}, \avg{\nabla w}}_{\Ehi\cup\EhD} } \\
\notag &\quad + \LRa{ \gammakap h_e^{-1-\alpha} \jump{v}, \jump{w} }_{\Ehi} + \LRa{ \gammakap h_e^{-1 } \jump{v}, \jump{w} }_{\EhD} 
}
for $\alpha \ge 0$. Here $\nabla$ for $v \in \Vh$ is the element-wise gradient operator. If $\alpha = 0$, this is the bilinear form of the symmetric interior penalty discontinuous Galerkin method (SIPG), which is used in the enriched Galerkin (EG) method in \cite{Lee-Lee-Wheeler-2016}.
In this paper we are interested in the cases $\alpha \ge 1$ such that the interior penalization terms are overly {\color{black}{penalized}}. We will call these methods interior {\color{black}{over-penalized}} enriched Galerkin methods ({\color{black}{IOP-EG}}). For example, if $\alpha = 1$, then $a_h$ is 
\algns{ 
a_h (v, w) &= \LRp{ \nabla v, \nabla w} - \LRp{ \LRa{ \avg{\nabla v}, \jump{w}}_{\Ehi\cup\EhD} + \LRa{ \jump{v}, \avg{\nabla w}}_{\Ehi\cup\EhD} } \\
\notag &\quad + \LRa{ \gammakap h_e^{-2} \jump{v}, \jump{w} }_{\Ehi} + \LRa{ \gammakap h_e^{-1} \jump{v}, \jump{w} }_{\EhD}. 
}

In {\color{black}{IOP-EG}} methods we seek $\uh \in \Vh$ such that 
\algn{ \label{eq:eg-eq}
a_h(\uh, v) = \F(v) \qquad \forall v \in \Vh 
}
for $a_h$ in \eqref{eq:ah} with
\algn{ \label{eq:F}
	\F (v) := (f, v) + \LRa{u_N, v }_{\Gamma_N} + \LRa{u_D, \n \cdot \nabla v }_{\Gamma_D} + \LRa{ \gammakap h_e^{-1} u_D, v}_{\Gamma_D} .
}
To show the stability of \eqref{eq:eg-eq}, let us define a norm on $\Vh$ associated with $a_h$ as 
\algn{ \label{eq:ah-norm}
\norw{ v }{a_h}^2 
&:= \LRp{ \nabla v, \nabla v} + \LRa{ \gammakap h_e^{-1-\alpha} \jump{v}, \jump{v} }_{\Ehi} + \LRa{ \gammakap h_e^{-1} \jump{v}, \jump{v} }_{\EhD}  .
}
For later use, we also define 
\algn{
  \norw{ v }{a_h, 0}^2 &:= \LRp{ \nabla v, \nabla v} + \LRa{ \gammakap h_e^{-1-\alpha} \jump{v}, \jump{v} }_{\Ehi}, & \norw{ v }{a_h, \pd}^2 &:=  \LRa{ \gammakap h_e^{-1} \jump{v}, \jump{v} }_{\EhD}.
}
By the inverse trace inequality and the standard argument in DG methods, it is not difficult to show
\algns{
| a_h(v, w) | \lesssim \norw{v}{a_h} \norw{w}{a_h}
}
for $v, w \in \Vh$. We can obtain that 
\algn{ \label{eq:cross-bound}
| \LRa{ \avg{\nabla v}, \jump{v}}_{\Ehi\cup\EhD} | \le \frac 12 \LRp{ \nabla v, \nabla v}^{\half} \LRa{ \gammakap h_e^{-1} \jump{v}, \jump{v} }_{\Ehi \cup\EhD}^{\half} 
}
for sufficiently large $\gamma$. 
Since $h \le 1$, if $\gamma$ is sufficiently large, then we can use \eqref{eq:cross-bound} to derive a coercivity inequality
\algn{ \label{eq:coercive}
\Cc \norw{v}{a_h}^2 \le a_h(v, v) 
}
with $\Cc >0$ independent of $0 < h \le 1$. Therefore \eqref{eq:eg-eq} is stable.

\section{The a priori error estimates}
\label{sec:apriori}
In this section we prove the a priori error analysis of the {\color{black}{IOP-EG}} methods.
For the analysis we adopt the idea of \cite{Gudi-2010} utilizing estimates of the a posteriori error analysis. 
As a consequence, we prove optimal a priori error estimates with minimal regularity assumption of exact solutions.

We recall efficiency estimates of the a posteriori error analysis in Nitsche's method \cite{Nitsche,Luthen-Juntunen-Stenberg-15}. 
We take into account only the data oscillation of $f$ in the discussion below for simplicity. 
For detailed discussion with data oscillations of boundary conditions, we refer to \cite{Luthen-Juntunen-Stenberg-15}.
For $v \in \Vhc$ the local efficiency estimates 
\algn{
\label{eq:aposteriori-1} h_T^2 \nor{f + \Delta v}_{0,T}^2 &\lesssim (\nabla (u - v), \nabla(u-v))_{0,T}^2 + (\osc_T(f))^2 ,\\
\label{eq:aposteriori-2} h_e \nor{\jump{\nabla v} }_{0,e}^2 &\lesssim \sum_{T \in \mc{T}_h, e \subset T} \LRp{(\nabla (u - v), \nabla(u-v))_{0,T}^2 + (\osc_T(f))^2 }
}
are proved in \cite{Luthen-Juntunen-Stenberg-15} with 
\algns{
\osc_T(f) &:= h_T \nor{f - f_h }_{0,T} , & \quad 
\osc(f) &:= \LRp{ \sum_{T \in \mc{T}_h} (\osc_T(f) )^2 }^\half 
}
where $f_h$ is the $L^2$ projection of $f$ into $\Vhc$.

Here we give some definitions and results which are necessary in our a priori error analysis.
We will use $\Vhnn$ to denote the space of mean-value zero piecewise constant functions, i.e., 
\algns{ 
\Vhnn := \{ v \in \Vhn \; : \; \int_\Omega v \,{\rm d} x = 0 \} .
}
Suppose that $D$ is a union of (closed) simplices in $\mc{T}_h$ and $\inte D$ is the interior of ${D}$. 
A discrete seminorm $| v |_{H_h^1(D)}$ for element-wise $H^1$ function $v$ is defined as  
\algn{
| v |_{H_h^1(D)}^2 = \| \nabla v \|_{L^2(D)}^2 + \sum_{e \in \Eh^0, e \subset \inte D} h_e^{-1} \| \jump{v} \|_{0,e}^2 .
}
If $D = \overline{\Omega}$, we simply use $H_h^1$ instead of $H_h^1(\overline{\Omega})$.

For a vertex $z$ in the triangulation $\mc{T}_h$, $T \in \mc{T}_h$, $e \in \Eh$, we define
\algn{ \label{eq:macroelement}
	M_z = \bigcup_{z {\color{black}{\in}} T} \{ T \in \mc{T}_h \,:\, z \in T\}, \quad M_T = \bigcup_{z {\color{black}{\in}}  T} M_z, \quad M_e = \bigcup_{T \in \mc{T}_h, e \subset \pd T} M_T .
}
Geometrically, these are the union of simplices which become proper neighborhoods whose interiors contain $z$, $T$, and $e$, respectively. 
For $v \in \Vhn$ it is known that there is a linear interpolation operator $I_h : \Vhn \ra \Vhc $ such that,   
\algn{
\label{eq:Ih-estm1} \norw{v - I_h v}{L^2(T)} &\lesssim h_T |v|_{H_h^1(M_T)} & & \forall T \in \Th, \\
\label{eq:Ih-estm2} \norw{v - I_h v}{L^2(e)} &\lesssim h_e^{\half} |v|_{H_h^1(M_e)} & & \forall e \in \Ehi \cup \EhD , \\
\label{eq:Ih-estm3} \norw{\nabla I_h v}{L^2(T)} &\lesssim |v|_{H_h^1(M_T)} & & \forall T \in \Th 
}
hold, where the implicit constants in these estimates are independent of the mesh sizes (see \tred{\cite[Theorem~3.1]{Buffa-Ortner-2009}}).

We define $\Th^D$ and $\mc{E}_h(\Th^D)$ as 
\algn{
\Th^D &:= \{ T \in \Th \,:\, \pd T \cap \pd \Omega \subset \Gamma_D \}, \\
\mc{E}_h(\Th^D) &:= \{ e \in \mc{E}_h \,:\, e \subset \pd T, \text{ for some } T \in \mc{T}_h^D \}. 
}
\begin{lemma} \label{lemma:V0-bdy}
For $\vn \in \Vhnn$ it holds that 
\algn{ \label{eq:vn-bdy-vn-int}
\norw{\vn}{a_h, \pd} \le C \max_{e \in \EhD} \{ h_e^{-\half} \} \max_{e \in \Ehi} \{ h_e^{\alpha/2} \} \norw{\vn}{a_h, 0} 
}
with $C$ depending on the shape regularity, the constants of the trace theorem and the discrete Poincar\'e inequality.
\end{lemma}
\begin{proof}
Suppose that $T \in \Th^D$ and $e \in \EhD$ with $e \subset \pd T$, are fixed. If we set $z \in \Vhnn$ as $z = h_e^{-1/2} \vn$, then, by \eqref{eq:Ih-estm2}, 
\algns{
\LRa{ h_e^{-1} \vn, \vn}_e = \norw{z}{L^2(e)}^2 &\le 2 \norw{I_h z}{L^2(e)}^2 + C h_e | z |_{H_h^1 (M_e)}^2 \\
&\le 2 h_e^{-1} \norw{I_h \vn}{L^2(e)}^2 + C  |\vn|_{H_h^1 (M_e)}^2  .
}
Applying a similar argument for every such pair $(T,e)$, $T \in \mc{T}_h^D$, 
\algn{ \label{eq:vn-ineq}
\norw{\vn}{a_h, \pd}^2 \le 2 \max_{e \in \EhD} \{ h_e^{-1} \} \gamma \norw{I_h \vn}{L^2(\Gamma_D)}^2 + C \gammakap  |\vn|_{H_h^1}^2 .
}
By the trace theorem and the triangle inequality, 
we have 
\algns{
\norw{I_h \vn}{L^2(\Gamma_D)} &\lesssim \norw{I_h \vn}{L^2(\Omega)} + \norw{\nabla I_h \vn}{L^2(\Omega)} \\
&\le \norw{I_h \vn - \vn }{L^2(\Omega)} + \norw{\vn }{L^2(\Omega)} + \norw{\nabla I_h \vn}{L^2(\Omega)} .
}
By a discrete Poincar\'e inequality (cf. \cite[Remark~1.1]{Brenner-poincare}), $\norw{\vn }{L^2(\Omega)} \lesssim | \vn |_{H_h^1}$ because $\vn \in \Vhnn$. Applying this discrete Poincar\'{e} inequality, \eqref{eq:Ih-estm1}, and \eqref{eq:Ih-estm3}, to the above inequality, we get 
\algns{
\norw{I_h \vn}{L^2(\Gamma_D)} \lesssim | \vn |_{H_h^1} .
}
From this, \eqref{eq:vn-ineq}, and 
$| \vn |_{H_h^1} \le \max_{e \in \mc{E}_h^0} \gammakap^{-1/2} h_e^{\alpha/2} \norw{\vn}{a_h, 0}$,
we can obtain 
\algns{
\norw{\vn}{a_h, \pd} \le C \max_{e \in \EhD} \{ h_e^{-\half} \} \max_{e \in \Ehi} \{ h_e^{\alpha/2} \} \norw{\vn}{a_h, 0} .
}
\end{proof}

In the theorem below, we assume that $\max_{e \in \EhD} \{ h_e^{-\half} \} \max_{e \in \Ehi} \{ h_e^{\alpha/2} \}$ is bounded by $Ch^{(\alpha-1)/2}$ with $C>0$ independent of the mesh sizes.
This condition is fulfilled, for example, if $\mathcal{T}_h$ is quasi-uniform.
Under this assumption, we prove the a priori error estimates for solutions $u \in H^s(\Omega)$, $s \ge 1$, of \eqref{eq:poisson-eq}--\eqref{eq:poisson-bc}.
\begin{theorem}
Suppose that $u$ and $\uh$ are the solutions of \eqref{eq:poisson-eq}--\eqref{eq:poisson-bc} and \eqref{eq:eg-eq}. Assume that $\alpha \ge 1$ and 
\algns{
\max_{e \in \EhD} \{ h_e^{-\half} \} \max_{e \in \Ehi} \{ h_e^{\alpha/2} \} \le C h^{(\alpha-1)/2} 
}
with $C>0$ independent of the mesh sizes. Then, 
\algn{ \label{eq:u-ah-estm}
\norw{u - \uh}{a_h} \le C (\inf_{v \in \Vhc} \norw{u - v}{1} + \osc(f) ) 
}
with $C>0$ independent of the mesh sizes. In addition, the interior jump terms satisfy
\algn{ \label{eq:u0-jump-estm}
| \uhn |_{H_h^1} \lesssim h^{\alpha} \LRp{ \inf_{v \in \Vhc} \norw{u - v}{1} + \osc(f) } .
}
{\color{black}{We remark that \eqref{eq:u-ah-estm} can be replaced by $\norw{u - \uh}{a_h} \le C h^s \norw{u}{s+1}$, a more common form of error estimate, if $1 \le s \le k+1$ because $\norw{f}{s-1} \le C \norw{u}{s+1}$.}}
\end{theorem}

\begin{proof}
 
We define $\Pic : H^1 (\Omega) \ra \Vhc$ by the linear map seeking $\whc \in \Vhc$ such that 
\algn{ \label{eq:Pic-proj}
a(\whc, v) = \F(v) \qquad \forall v \in \Vhc .
}
Further, we define $\Pin : (\Vhnn)^* \ra \Vhnn$ by the linear map seeking the solution $\whn \in \Vhnn$ of 
\algn{ \label{eq:Pi0-proj}
a_h(\whn, v) = a_h(\Pic u, v)  - \F(v) \qquad \forall v \in \Vhnn .
}

We first claim that for $\vn \in \Vhnn$
\algn{ \label{eq:Pi0u-v-estm}
|a_h(\Pin u, \vn)| \lesssim \LRp{ \inf_{w \in \Vhc} \norw{u - w}{1} + \osc(f) } |\vn|_{H_h^1} .
}
To prove it, we use \eqref{eq:Pi0-proj}, the Galerkin orthogonality, and the integration by parts to get
\algns{
a_h(\Pin u, \vn) 
&= a_h(\Pic u, \vn - I_h \vn)  - \F(\vn - I_h \vn) \\
&= - \sum_{T \in \Th} (f + \Div (\nabla \Pic u), \vn - I_h \vn) \\
&\quad - \sum_{e \in \Ehi \cup \EhD} \LRa{\jump{\nabla \Pic u }, {\vn - I_h \vn} }_e .
}
By \eqref{eq:Ih-estm1} and \eqref{eq:Ih-estm2}, we have 
\algns{
\LRp{ f + \Delta \Pic u, \vn - I_h \vn}_T &= \norw{f + \Delta \Pic u}{0,T} \norw{\vn - I_h \vn}{0,T} \\
&\lesssim h_T \norw{f + \Delta \Pic u}{0,T} |\vn|_{H_h^1(M_T)} , \\
\LRa{ \jump{\nabla \Pic u }, \vn - I_h \vn}_e &\le \norw{\jump{\nabla \Pic u }}{0,e} \norw{\vn - I_h \vn}{0,e} \\
&\lesssim h_e^{\half} \norw{\jump{\nabla \Pic u }}{0,e} |\vn|_{H_h^1(M_e)} .
}
The elements in $\{ M_T\}$ and $\{M_e\}$ overlap only finitely many times with a uniform finite number, so
the Cauchy--Schwarz inequality and the a posteriori error estimates \eqref{eq:aposteriori-1}--\eqref{eq:aposteriori-2} yield \eqref{eq:Pi0u-v-estm}. 

Note that 
\algn{ \label{eq:H1-ah-estm}
| \vn |_{H_h^1}^2 \lesssim h^{\alpha} \norw{ \vn}{a_h, 0}^2 \lesssim h^{\alpha} \norw{ \vn}{a_h}^2 .
}
By this and by taking $\vn = \Pin u$ in \eqref{eq:Pi0u-v-estm}, 
\algns{
h^{-\alpha} | \Pin u |_{H_h^1}^2 \lesssim \norw{\Pin u}{a_h, 0}^2 \le  \norw{\Pin u}{a_h}^2 \lesssim | \Pin u |_{H_h^1} \LRp{ \inf_{w \in \Vhc} \norw{u - w}{1} + \osc(f) } ,  
}
so we have 
\algn{ 
\label{eq:Pi0u-estm} | \Pin u |_{H_h^1} &\lesssim h^{\alpha} \LRp{ \inf_{w \in \Vhc} \norw{u - w}{1} + \osc(f) } , \\
\label{eq:Pi0u-estm2} \norw{\Pin u}{a_h} &\lesssim h^{\alpha/2} \LRp{ \inf_{w \in \Vhc} \norw{u - w}{1} + \osc(f) } .
}
We now define $\Pi_h : H^1(\Omega) \ra \Vh$ as $\Pi_h = \Pic + \Pin$. Then
\algns{
&a_h( \uh - \Pi_h u, \uh - \Pi_h u) \\
&\quad = \F( \uh - \Pi_h u) - a_h( \Pi_h u, \uh - \Pi_h u ) \\
&\quad = \F(\uhc - \Pic u) + \F(\uhn - \Pin u) - a_h(\Pic u, \uhc - \Pic u) \\ 
&\qquad  - a_h(\Pin u, \uhc - \Pic u) - a_h(\Pic u, \uhn - \Pin u) - a_h(\Pin u, \uhn - \Pin u). 
}
Since 
\algns{ 
\F(\uhc - \Pic u) - a_h(\Pic u, \uhc - \Pic u) &= 0, \\
\F(\uhn - \Pin u) - a_h( \Pic u , \uhn - \Pin u ) &= - a_h(\Pin u, \uhn - \Pin u ) ,
}
by \eqref{eq:Pic-proj} and \eqref{eq:Pi0-proj}, we have 
\algn{
&|a_h( \uh - \Pi_h u, \uh - \Pi_h u) | \notag \\
&\quad = | a_h (\Pin u, (\Pic u - \uhc) ) + 2 a_h (\Pin u, (\Pin u - \uhn) ) | \notag \\ 
&\label{eq:ah-estm} \quad = |a_h (\Pin u, \Pi u - \uh ) + a_h(\Pin u, \Pin u - \uhn) | \\
&\quad \lesssim \norw{\Pin u}{a_h} \norw{\Pi_h u - \uh}{a_h} + \norw{\Pin u}{a_h} \norw{\Pin u - \uhn}{a_h}  \notag .
}
By Lemma~\ref{lemma:V0-bdy}, $\norw{\Pi_h^0 u - u_h^0}{a_h,\pd} \lesssim h^{(\alpha-1)/2}\norw{\Pi_h^0 u - u_h^0}{a_h,0}$. Since 
\algns{
\norw{\Pi_h^0 u - u_h^0}{a_h,0} \le \norw{\Pi_h u - u_h}{a_h,0} \le \norw{\Pi_h u - u_h}{a_h},
}
we have 
\algns{
\norw{\Pi_h^0 u - u_h^0}{a_h} \le \norw{\Pi_h^0 u - u_h^0}{a_h, 0}  + \norw{\Pi_h^0 u - u_h^0}{a_h, \pd}  \lesssim \norw{\Pi_h u - u_h}{a_h} ,
}
therefore we can obtain 
\algns{
  \norw{\uh - \Pi_h u}{a_h} \lesssim \norw{\Pin u}{a_h}
}
from \eqref{eq:ah-estm}. If we use \eqref{eq:Pi0u-estm2}, 
then one can obtain with the coercivity of $a_h$ that
\algns{
\norw{\uh - \Pi_h u}{a_h} \lesssim  h^{\alpha/2} \LRp{\inf_{v \in \Vhc} \norw{u - v}{1} + \osc(f) } .
}
In particular, by $| \uhn - \Pi_h^0 u |_{H_h^1} \lesssim h^{\alpha/2} \norw{\uhn - \Pi_h^0 u}{a_h,0} \lesssim h^{\alpha/2} \norw{\uh - \Pi_h u}{a_h,0}$, we obtain
\algns{
 | \uhn - \Pi_h^0 u |_{H_h^1} \lesssim h^{\alpha} \LRp{\inf_{v \in \Vhc} \norw{u - v}{1} + \osc(f) } ,
}
therefore \eqref{eq:u0-jump-estm} follows from this, \eqref{eq:Pi0u-estm}, and the triangle inequality.

To complete the proof we estimate $\norw{u - \Pi_h u}{a_h}$. By the triangle inequality and \eqref{eq:Pi0u-estm2}, 
\algns{
\norw{u - \Pi_h u}{a_h} &\le \norw{ u - \Pic u}{a_h} + \norw{\Pin u}{a_h} \lesssim \inf_{w \in \Vhc} \norw{u - w}{1} + \osc(f)
}
where we used the estimate $\norw{ u - \Pic u}{a_h} \lesssim \inf_{w \in \Vhc} \norw{u - w}{1} + \osc(f)$ which holds because $\Pic u$ is 
the solution of Nitsche's method of the problem \eqref{eq:Pic-proj}.
\end{proof}

Recall that the flux is $\zb := - \nabla u$ by Darcy's law. 
We now discuss recovery of locally mass conservative numerical flux $\zb_h$ via local post-processing
and prove an estimate of $\zb - \zb_h$. 
The key result is that the estimate is robust for the over-penalization.

We use $H(\div;\Omega)$ to denote the subspace of $L^2(\Omega; \R^n)$ such that the function and its divergence are square-integrable.
\begin{definition}
We say $\zb_h \in H(\div;\Omega)$ is locally mass conservative if 
\algns{
(\div \zb_h , 1_T) = (f , 1_T)
}
for any $T \in \mc{T}_h$ where $1_T$ is the indicator function on $T$. 
\end{definition}

For construction of locally mass conservative numerical flux we use the Raviart--Thomas--N\'{e}d\'{e}lec  element of order $k$, $N_h$, which has local shape functions 
\algns{
\mc{P}_{k-1}(T; \R^n) + \bs{x} \mc{P}_{k-1}(T)
}
where $\bs{x}$ is the vector-valued polynomial $\bs{x} = (x \;y)^T$ if $n=2$ and $\bs{x} = (x \;y \;z)^T$ if $n=3$. 
\begin{theorem} 
	\label{thm:flux-error}
If we define $\zb_h \in N_h$ as 
\algns{
(\zb_h, \bs{q})_T &= (- \nabla \uh , \bs{q})_T, &  & \bs{q} \in \Pb_{k-2}(T), \\
\LRa{ \zb_h \cdot \n, q }_e &= \LRa{- \avg{\nabla \uh} \cdot \n + \gammakap h_e^{-1-\alpha} \jump{\uh} \cdot \n , q }_e, & & q \in \mc{P}_{k-1}(e), e \in \Ehi , \\
\LRa{ \zb_h \cdot \n, q }_e &= \LRa{u_N , q }_e, & & q \in \mc{P}_{k-1}(e), e \not\in \Ehi \cup \EhD, \\
\LRa{ \zb_h \cdot \n, q }_e &= \LRa{- {\nabla \uh} \cdot \n + \gammakap h_e^{-1} (\uh - u_D) , q }_e, & & q \in \mc{P}_{k-1}(e), e \in \EhD , \\
}
then 
\algns{
\norw{\zb_h + \nabla \uh}{0} \lesssim \inf_{w \in \Vhc} \norw{u - w}{1} + \osc(f) 
}
with an implicit constant independent of $h$.
\end{theorem}
\begin{proof}
By definition $\zb_h$ is in $H(\div;\Omega)$ and 
\algns{
(\div \zb_h, 1_T)  &= \sum_{e \subset \pd T, e \in \EhD} \LRa{- \nabla \uh \cdot \n + \gammakap h_e^{-1} (\uh - u_D) , 1_T}_e \\
&\quad + \sum_{e \subset \pd T, e \in \Ehi} \LRa{- \avg{\nabla \uh} + \gammakap h_e^{-1-\alpha} \jump{\uh} \cdot \n, 1_T}_e  \\
&\quad + \sum_{e \subset \pd T, e \not \in \Ehi \cup \EhD} \LRa{u_N, 1_T}_e .
}
If we take $v \in \Vh$ in \eqref{eq:eg-eq} as the indicator function $1_T$ and use the definition of $a_h$ in \eqref{eq:ah}, then
\algns{
(\div \zb_h, 1_T) = (f, 1_T) ,
}
so $\zb_h$ is locally mass conservative.

From the definition of $\zb_h$, $\zb_h + \nabla \uh$ satisfies
\algns{
({\zb}_h + \nabla \uh, \bs{q})_T &= 0 , &  & \bs{q} \in \Pb_{k-2}(T),
}
and 
\algn{ \label{eq:flux-case}
\LRa{ ({\zb}_h + \nabla \uh) \cdot \n, q}_e &= 
\case{
\LRa{-\half \jump{\nabla \uh} + \gammakap h_e^{-1-\alpha} \jump{\uh} \cdot \n, q }_e, & e \in \Ehi \\ 
0, & e \not \in \Ehi \cup \EhD \\
\LRa{ \gammakap h_e^{-1} (\uh - u_D) , q }_e, & e \in \EhD 
} 
}
for $q \in \mc{P}_{k-1}(e)$. For $e \in \Ehi$, the standard scaling argument gives 
\algns{
\norw{{\zb}_h + \nabla \uh}{0,T} &\lesssim \sum_{e \subset \pd T} h_e^{\half} \norw{({\zb}_h + \nabla \uh) \cdot \n }{0,e} .
}
We now estimate the above term for three cases of $e \in\Eh$. 
If $e \not \in \Ehi \cup \EhD$, then there is nothing to estimate by \eqref{eq:flux-case}. 
If $e \in \Ehi$, then 
\algns{
\norw{{\zb}_h + \nabla \uh}{0,T} &\lesssim \sum_{e \subset \pd T} h_e^{\half}\norw{-\half \jump{\nabla \uh} + \gammakap h_e^{-1-\alpha} \jump{\uh} \cdot \n }{0,e} \\
&\lesssim \sum_{e \subset \pd T} h_e^{\half} \LRp{ \norw{\jump{\nabla \uh} }{0,e} + \gammakap h_e^{-1-\alpha} \norw{\jump{\uh} }{0,e} } ,
}
so we need to show 
\mltln{ \label{eq:flux-aux-estm}
  \LRp{ \sum_{e \in \Ehi} h_e \| \jump{\nabla u_h} \|_{0,e}^2 }^{\frac 12} + \LRp{ \sum_{e \in \Ehi} h_e^{-1-2\alpha} \| \jump{u_h} \|_{0,e}^2 }^{\frac 12} \\ 
  \lesssim \inf_{v \in \Vhc} \norw{u - v}{1} + \osc(f) .
}
The first term in this inequality is easily obtained by \eqref{eq:aposteriori-2}. For the second term note that the quasiuniformity assumption implies 
\algns{
\sum_{e \in \Ehi } h_e^{-1-2\alpha} \norw{\jump{\uh}}{0,e}^2  \lesssim h^{-2\alpha} | \uhn |_{H_h^1}^2 \lesssim \LRp{\inf_{v \in \Vhc} \norw{u - v}{1} + \osc(f) }^2 
}
by \eqref{eq:u0-jump-estm}. 
For $e \in \EhD$, the scaling argument gives
\algns{
	\tred{\norw{{\zb}_h + \nabla \uh}{0,T}} &\lesssim \sum_{e \subset \pd T} h_e^{\half} \norw{ \gammakap h_e^{-1} (\uh - u_D) }{0,e} \lesssim \norw{ u - \uh}{a_h} .
}
Therefore, the conclusion follows. 
\end{proof}

\section{Preconditioning}
\label{sec:preconditioning}
In this section we propose an abstract form of block diagonal preconditioners. \tred{Through} the operator preconditioning approach, we show that the preconditioners are spectrally equivalent to the matrix given by \eqref{eq:eg-eq} independent of mesh sizes.

We define a norm $\tnorm{\cdot}$ on $\VBh$ as 
\algn{ 
	\notag \tnorm{ (\vc, \vn) }^2 &:= a_h(\vc, \vc) + a_h(\vn, \vn) \\
	\label{eq:tnorm} &= \tred{\LRp{ \nabla \vc, \nabla \vc}} + \LRa{ \gammakap h_e^{-1} \vc , \vc }_{\mc{E}_h^{D}} \\ 
	\notag &\quad + \LRa{ \gammakap h_e^{-1-\alpha} \jump{{\vn}},\jump{{\vn}} }_{\Ehi } + \LRa{ \gammakap h_e^{-1} \vn , \vn }_{\EhD} 
}
for $\vc \in \Vhc$, $\vn \in \Vhn$. Note that this is a norm on $\VBh$, not only on $\Vh$. For an element $v \in \Vh$ the decomposition $v = \vc + \vn$ is not unique, so $\tnorm{v}$ is not uniquely defined. However, we will use $\tnorm{\vc}$ and $\tnorm{\vn}$ for $\vc \in \Vhc$ and $\vn \in \Vhn$ instead of $\tnorm{(\vc, 0)}$ and $\tnorm{(0,\vn)}$ for convenience
%
%
even though these quantities are not norms.

Recall that $\Vhperp$ and $\Vh$ have a one-to-one correspondence, so we will use $v$ to denote an element in $\Vh$ and its corresponding element in $\Vhperp$. Let ${\mathcal{A}_h} : \Vhperp \ra (\Vhperp)^*$ be the linear operator defined by 
$\LRa{{\mathcal{A}_h} v , w}_{\LRa{\VBh, \VBh^*}} = a_h(v, w)$. 
Following the operator preconditioning approach \cite{Mardal-Winther-2011}, we first show that ${\mathcal{A}_h}$ is a bounded isomorphism from $\Vhperp$ to $(\Vhperp)^*$ with the norm $\tnorm{\cdot}$.

For preconditioning we consider the linear operator $\tilde{{\mathcal{A}_h}} : \VBh \ra \VBh^*$ defined by the bilinear form
\algns{
 a_h(\vc, \vc) + a_h(\vn, \vn) ,
}
and show that $\tilde{{\mathcal{A}_h}}$ is spectrally equivalent to ${\mathcal{A}_h}$ on $\Vhperp$ with $\tnorm{\cdot}$.
As a consequence, we can expect that $\mc{B}_h = \tilde{{\mathcal{A}_h}}^{-1}$ is a good preconditioner of ${\mathcal{A}_h}$. 

It is obvious that ${\mathcal{A}_h}$ gives a bounded bilinear form on $\Vhperp \times \Vhperp$ with $\tnorm{\cdot}$ norm.
To show that $\tilde{{\mathcal{A}_h}}$ is spectrally equivalent to ${\mathcal{A}_h}$ on $\Vhperp$, we prove 
\algns{
\inf_{v = \vc + \vn \in \Vh} \sup_{w = \wc + \wn \in \Vh} \frac{a_h(v, w)}{\tnorm{(\vc, \vn)} \tnorm{(\wc, \wn)}} \ge C > 0 .
}
We remark that the coercivity of $a_h(\cdot, \cdot)$ for $\norw{\cdot}{a_h}$ does not imply this condition because $\norw{v}{a_h} \lesssim \tnorm{(\vc, \vn)}$ for $v = \vc + \vn$ but the inequality of the other direction is not true in general.

\begin{definition}
For $T \in \Th$, we define $N(T)$ as the number of elements adjacent to $T$, i.e., 
$N(T) = | \{ T' \in \Th \,:\, \pd T' \cap \pd T \not = \emptyset, T' \not = T \} |$.
For $\vn \in \Vhn$, $T_0 \in \mc{T}_h$, and the elements $\{T_i \}_{i=1,..., N(T_0)}$ adjacent to $T_0$, let 
$p_i$ be the value of $\vn|_{T_i}$, $i = 0, ..., N(T_0)$.
We define $\phi_{\vn, T_0} \in \Vhn$ as 
\algn{ \label{eq:phi}
\phi_{\vn, T_0} |_{T} = \case{ \frac{1}{N(T_0)+1} \sum_{i=0}^{N(T_0)} (p_i - p_0) \qquad &\text{if } T = T_0  \\ 
0 &\text{otherwise } 
}.
}
\end{definition}

\begin{lemma} \label{lemma:vwtil}
Suppose that $\alpha \ge 1$. Given $\vn \in \Vhnn$, define $\bar{v}^0$ as $\bar{v}^0:= \sum_{T \in \Th^D} \phi_{\vn, T}$. Then 
\algn{ \label{eq:wtil-estm}
\norw{\bar{v}^0}{a_h,\pd} \le \tnorm{\bar{v}^0} \le C_0 \norw{\vn}{a_h,0} . 
}
with $C_0>0$ independent of $h$. 
As a corollary, $\vtiln := \vn + \bar{v}^0$ satisfies 
\algn{ 
\label{eq:vwtil-estm} &\norw{\vtiln}{a_h,0} \le (1 + C_0) \norw{\vn}{a_h, 0}, \\
\label{eq:vwtil-estm2} &\norw{\vtiln }{a_h, \pd} \lesssim \norw{\vn }{a_h, 0}
}
with an implicit constant depending on the shape regularity of $\Th$, the constants of the trace theorem and discrete Poincar\'{e} inequality.
\end{lemma}
\begin{proof}
Let $T_0$ be an element in $\mc{T}_h^D$. For simplicity of presentation we assume that $T_0$ has only one facet in $\pd T_0 \cap \Gamma_D$
but extension of the arguments below to more general cases is straightforward. 

Assume that $e_0 = \pd T_0 \cap \Gamma_D$, and 
$\{T_i\}_{i=1}^{N(T_0)}$ are the elements adjacent to $T_0$ such that $e_i = \pd T_i \cap \pd T_0$.
Denoting the value of $v^0$ on $T_i$ by $p_i$, $i = 0, ..., N(T_0)$, 
$|\jump{\vn}|$ on each $e_i$ is $|p_i - p_0|$.
By the definition of $\phi_{\vn, T_0}$ 
\algns{
| \jump{\phi_{\vn, T_0}}|_{e_j} | = \left| \frac{1}{N(T_0)+1} \sum_{i=1}^{N(T_0)} (p_i - p_0)\right| \le  \frac{1}{N(T_0)+1} \sum_{i=1}^{N(T_0)} | p_i - p_0 | 
}
for $0 \le j \le N(T_0)$, so the Cauchy--Schwarz inequality and the definition of $\tnorm{\cdot}$ give
\algns{
\tnorm{ \phi_{\vn, T_0} }^2 &= \sum_{e \subset \pd T_0, e \in \Ehi } \gamma h_e^{-1 - \alpha} \norw{ \jump{\phi_{\vn, T_0}} }{0,e}^2 +
\gamma h_{e_0}^{-1} \norw{ \phi(\vn, T_0)}{0,e_0}^2 \\
&\lesssim \gamma \sum_{i=1}^{N(T_0)} h_{e_i}^{-1-\alpha} \norw{ \jump{\vn} }{e_i}^2 , \\
\norw{ \phi_{\vn, T_0} }{a_h,\pd}^2 &= \gamma h_{e_0}^{-1} \norw{ \phi(\vn, T_0)}{0,e_0}^2  \lesssim h^{\alpha} \sum_{i=1}^{N(T_0)} \gamma h_{e_i}^{-1-\alpha} \norw{  \jump{\vn} }{0,e_i}^2 .
}
By the definition of $\bar{v}^0$, $\bar{v}^0|_T = \phi(\vn, T) $ for all $T \in \mc{T}_h^{D}$ and $\bar{v}^0|_T = 0$ for $T \not \in \mc{T}_h^D$.
Since each facet is shared at most by two elements, the triangle inequality with the above argument gives 
\algn{ \label{eq:wtil-estm0}
\tnorm{\bar{v}^0}^2 \lesssim  \norw{\vn}{a_h,0}^2 ,
}
which also implies $\norw{\bar{v}^0}{a_h,\pd}^2 \lesssim \norw{\vn}{a_h,0}^2$.
Then, \eqref{eq:vwtil-estm} follows by the triangle inequality $\norw{\vtiln}{a_h,0} \le \norw{\vn}{a_h,0} + \norw{\bar{v}^0}{a_h,0}$ and the inequality $\norw{\bar{v}^0}{a_h,0} \le \tnorm{\bar{v}^0}$. \eqref{eq:vwtil-estm2} follows by  $\norw{\vtiln}{a_h,\pd} \le \norw{\vn}{a_h,\pd} + \norw{\bar{v}^0}{a_h,\pd}$, \eqref{eq:vn-bdy-vn-int} and \eqref{eq:wtil-estm0}. 
\end{proof}

\begin{theorem} \label{thm:precond-estm}
Assume that the assumptions in Lemma~\ref{lemma:vwtil} hold. If $\alpha \ge 1$, then there exists $C>0$ such that 
\algn{ \label{eq:inf-sup}
\inf_{v = \vc + \vn \in \Vh} \sup_{w = \wc + \wn \in \Vh} \frac{a_h(v, w)}{\tnorm{(\vc, \vn)} \tnorm{(\wc, \wn)}} \ge C > 0 .
}
\end{theorem}
\begin{proof}
For $v \in \Vh$ let $v = \vc + \vn$ for $\vc \in \Vhc$ and $\vn \in \Vhnn$. Suppose that $\bar{v}^0$ and $\vtiln$ are defined as in Lemma~\ref{lemma:vwtil}. 
By \eqref{eq:wtil-estm}, \eqref{eq:vwtil-estm}, \eqref{eq:vwtil-estm2}, and \eqref{eq:vn-bdy-vn-int} there exists $C_1$ independent of $h$ satisfying 
\algn{ \label{eq:aux-ineq1}
\norw{\bar{v}^0}{a_h,\pd}, \norw{\vtiln}{a_h, 0}, \norw{\vtiln}{a_h, \pd}, \norw{\vn}{a_h, \pd} \le C_1 \norw{\vn}{a_h, 0} .
}
We take 
\algns{
w := \wc + \wn, \qquad \wc := \vc, \quad  \wn := \delta \vtiln = \vn + \delta (\vn + \bar{v}^0)
}
with $\delta>0$ which will be determined later. From the definition \eqref{eq:tnorm} and \eqref{eq:wtil-estm}, it holds that 
\algn{ \label{eq:w-estm}
\tnorm{ (\wc, \wn)}^2 = \tnorm{(\vc, \vn + \delta \vtiln )}^2 = \tnorm{\vc}^2 + \tnorm{v + \delta \vtiln}^2 \leq C \tnorm{(\vc, \vn)}^2 
}
with $C$ which is uniformly bounded for $\delta \le 1$.

From the definition of $w$ and the coercivity of $a_h$
\algns{
a_h(v,w) &= a_h(v, v) + \delta a_h(v, \vtiln) \\
&\ge \Cc \LRp{ \norw{ \vc}{a_h,0}^2 + \norw{\vn}{a_h,0}^2 + \norw{\vn + \vc}{a_h,\pd}^2 } \\
&\quad - \delta \LRa{ \avg{\nabla \vc}, \jump{\vtiln}}_{\Eh^0 \cup \Eh^D } + \delta \LRa{ \gammakap h_e^{-1-\alpha } \jump{\vn}, \jump{\vtiln} }_{\mc{E}_h^0} \\
&\quad + \delta \LRa{ \gammakap h_e^{-1 } (\vn + \vc), \vtiln }_{\mc{E}_h^{D}} \\
&= \Cc \tnorm{(\vc, \vn)}^2 + 2 \Cc \LRa{\gamma h_e^{-1} \vc, \vn }_{\EhD} \\
&\quad - \delta \LRa{ \avg{\nabla \vc}, \jump{\vtiln }}_{\Eh^0 \cup \Eh^D } + \delta \LRa{ \gammakap h_e^{-1-\alpha } \jump{\vn}, \jump{\vtiln} }_{\mc{E}_h^0} \\
&\quad + \delta \LRa{ \gammakap h_e^{-1 } \vn, \vtiln }_{\mc{E}_h^{D}} + \delta \LRa{ \gammakap h_e^{-1 } \vc, \vn + \bar{v}^0}_{\mc{E}_h^{D}} \qquad (\text{because } \vtiln = \vn + \bar{v}^0)\\
&= \Cc \tnorm{(\vc, \vn)}^2 + (2 \Cc + \delta) \LRa{ \gammakap h_e^{-1} \vc, \vn }_{\mc{E}_h^{D}} + \delta \LRa{ \gammakap h_e^{-1} \vc, \bar{v}^0}_{\mc{E}_h^{D}} \\
&\quad - \delta \LRa{ \avg{\nabla \vc}, \jump{\vtiln}}_{\Eh^0 \cup \Eh^D} + \delta \LRa{ \gammakap h_e^{-1 -\alpha} \jump{\vn}, \jump{\vtiln} }_{\mc{E}_h^0} + \delta \LRa{ \gammakap  h_e^{-1} \vn , \vtiln }_{\mc{E}_h^{D}} \\
&=: \Cc \tnorm{(\vc, \vn)}^2 + (2 \Cc + \delta) I_1 + \delta (I_2 + I_3 + I_4 + I_5) 
}
with 
\algns{
  I_1 &:= \LRa{ \gammakap h_e^{-1} \vc, \vn }_{\mc{E}_h^{D}} & I_2 &:= \LRa{ \gammakap h_e^{-1} \vc, \bar{v}^0}_{\mc{E}_h^{D}} & I_3 &:= \LRa{ \avg{\nabla \vc}, \jump{\vtiln}}_{\Eh^0 \cup \Eh^D}, \\
  I_4 &:= \LRa{ \gammakap h_e^{-1 -\alpha } \jump{\vn}, \jump{\vtiln} }_{\mc{E}_h^0} & I_5 &:= \LRa{ \gammakap  h_e^{-1} \vn , \vtiln }_{\mc{E}_h^{D}} .
}
For simplicity we introduce additional notations 
\algns{
\Xc = \norw{\vc}{a_h, \pd}, \quad 
\Xn = \norw{\vc}{a_h, 0}, \quad 
\Yc = \norw{\vn}{a_h, \pd}, \quad 
\Yn = \norw{\vn}{a_h, 0} .
}
By the Cauchy--Schwarz inequality
\algns{
  |I_1| &\le \Xc \Yc \le C_1 \Xc \Yn, \quad \text{(by \eqref{eq:aux-ineq1})} \\
  |I_2| &\le \Xc \norw{\bar{v}^0}{a_h,\pd} \le C_1 \Xc \Yn , \quad \text{(by \eqref{eq:aux-ineq1})} \\
  |I_3| &\le C_2 \Xc (\Yc + \Yn) \le C_2 (1 + C_1) \Xc \Yn , \quad \text{(by \eqref{eq:cross-bound} and \eqref{eq:aux-ineq1}) } \\
  |I_4| &\le C_2 \Yn^2 . \quad \text{(by \eqref{eq:aux-ineq1})} 
}
Moreover, 
\algns{
I_5 = \Yc^2 + \LRa{\gamma h_e^{-1} \vn, \bar{v}^0 }_{\EhD} \ge \Yc^2 - C_1 \Yc \Yn . \quad \text{(by \eqref{eq:aux-ineq1})} 
}
If we use these inequalities to the previous form of $a_h(v, w)$, then 
\algns{
a_h(v, w) &\ge \Cc (\Xc^2 + \Xn^2 + \Yc^2 + \Yn^2 - 2 \Xc \Yc) \\
&\quad - \delta \LRp{ (3 C_1 + C_1 C_2) \Xc \Yn + C_2 \Yn^2 } + \delta \Yn^2 
}
By Young's inequality we can have 
\algns{
a_h(v, w) &\ge \Cc (\Xc^2 + \Xn^2 + \Yc^2 + \Yn^2 - 2 \Xc \Yc ) \\
&\quad - \delta^2 C_3 \Xc^2 - \frac {\Cc}4 \Yn^2 - \delta C_2 \Yn^2 + \delta \Yc^2 
}
with $C_3 >0$ depending on $C_1$, $C_2$, and $\Cc$. If we use 
\algns{
  2 \Xc \Yc \le \frac{1}{1 + \frac{\delta}2} \Xc^2 + \LRp{1 + \frac{\delta}2} \Yc^2, 
}
we can obtain
\algns{
a_h(v, w) &\ge \LRp{ \Cc \frac{\delta}{2+\delta} - 9 \delta^2 C_1^2} \Xc^2 + \Cc \Xn^2 + \frac{\Cc \delta}2 \Yc^2 + \frac {3\Cc}4  \Yn^2 .
}
If we choose $\delta$ sufficiently small, then 
\algns{
a_h(v, w) \ge C \tnorm{(\vc, \vn)}^2 .
}
The conclusion follows by combining it with \eqref{eq:w-estm}. 
\end{proof}

To construct preconditioners recall that $\tilde{{\mathcal{A}_h}}$ is given by the bilinear form 
\algn{ \label{eq:block-diag}
a_h (\vc, \wc) + a_h (\vn, \wn) , \qquad v, w \in \Vh ,
}
and its matrix form is 
\algn{ \label{eq:block-diag-mat}
\pmat{ \M_c & 0 \\
0 & \M_0
}
}
where $\M_c$ and $\M_0$ are matrices obtained from $a_h (\vc, \wc)$ and $a_h (\vn, \wn)$, respectively. 

To construct preconditioners in practice, we use algebraic multigrid methods to obtain 
an approximate inverse of this block diagonal matrix.
It is known that algebraic multigrid methods give good preconditioners for $\M_c$.
Since $\M_0$ is a weakly diagonally dominant matrix \cite{Lee-Lee-Wheeler-2016}, 
standard algebraic multigrid methods give an efficient preconditioner for this as well.
Therefore it is feasible to use a preconditioner of the form
\algn{ \label{eq:block-diag-mat-inv}
\pmat{ \text{AMG}(\M_c) & 0 \\
0 & \text{AMG}(\M_0)
} ,
}
where $\text{AMG}(\M)$ is a preconditioner of $\M$ constructed by algebraic multigrid methods.
We will see in the next section that this form of preconditioner gives robust numerical results.

\section{Numerical results}
\label{sec:numerical}

In this section we present results of numerical experiments. 
In all numerical experiments, $\Omega = [0,1]\times [0,1] \subset \R^2$
and meshes consist of triangles that are bisections of $N \times N$ subsquares of $\Omega$ ($N=4,8,16,32,64,128$).
All numerical experiments are implemented with Firedrake \cite{firedrake}. 
\begin{table}[h]
	{\color{black}{
	\begin{center} \tiny 
		\begin{tabular}{c|c|c|cc|cc|cc|cc} \hline
			\multirow{2}{*}{$\kappa_0$} & \multirow{2}{*}{$k$} & \multirow{2}{*}{$h_{\max}$} & \multicolumn{2}{c|}{$ \| u - u_h \|_{0} $} & \multicolumn{2}{c|}{$ \| u - u_h \|_{a_h} $} & \multicolumn{2}{c|}{$ \tred{\| \zb - \zb_h \|_{\kap^{-1}}} $} & \multicolumn{2}{c}{$ \| P_0(f - \div \zb_h) \|_0 $} \\
			 & & & error & rate & error & rate & error & rate & error & rate \\ \hline \hline 
			\multirow{12}{*}{1} & \multirow{6}{*}{1}  & 1/{4}  &  1.7308e-02  &   --  &  2.3161e-01  &   --  &  1.8076e-01  &   --  &  6.7771e-06  &   --    \\ 
                                & & 1/{8}  &  5.1275e-03  &  1.76 &  1.1741e-01  &  0.98 &  8.2400e-02  &  1.13 &  2.5351e-07  &  4.74   \\ 
                                & & 1/{16} &  1.3562e-03  &  1.92 &  5.8387e-02  &  1.01 &  3.7214e-02  &  1.15 &  9.4939e-10  &  8.06   \\ 
                                & & 1/{32} &  3.4623e-04  &  1.97 &  2.9030e-02  &  1.01 &  1.7672e-02  &  1.07 &  1.7626e-11  &  5.75   \\ 
                                & & 1/{64} &  8.7325e-05  &  1.99 &  1.4464e-02  &  1.01 &  8.6656e-03  &  1.03 &  7.8165e-13  &  4.50   \\ 
                                & & 1/{128}&  2.1919e-05  &  1.99 &  7.2184e-03  &  1.00 &  4.3051e-03  &  1.01 &  2.6110e-12  &  -1.74  \\ 
\cline{2-11} 
			& \multirow{6}{*}{2} &      1/{4}  &  9.4928e-04  &   --  &  3.2761e-02  &   --  &  2.2220e-02  &   --  &  3.4991e-06  &   --    \\ 
                                    & & 1/{8}  &  1.2497e-04  &  2.93 &  8.3853e-03  &  1.97 &  5.7538e-03  &  1.95 &  7.1763e-08  &  5.61 \\ 
                                    & & 1/{16} &  1.6095e-05  &  2.96 &  2.1117e-03  &  1.99 &  1.4424e-03  &  2.00 &  7.7163e-09  &  3.22 \\ 
                                    & & 1/{32} &  2.0441e-06  &  2.98 &  5.2921e-04  &  2.00 &  3.5990e-04  &  2.00 &  7.6639e-09  &  0.01 \\ 
                                    & & 1/{64} &  2.5760e-07  &  2.99 &  1.3242e-04  &  2.00 &  8.9842e-05  &  2.00 &  2.1702e-12  &  11.79\\ 
                                    & & 1/{128}&  3.2332e-08  &  2.99 &  3.3119e-05  &  2.00 &  2.2443e-05  &  2.00 &  5.6520e-12  &  -1.38\\ 
\hline		
			\multirow{12}{*}{10} & \multirow{6}{*}{1}  & 1/{4}  &  1.7137e-02  &   --  &  2.4350e-01  &   --  &  1.4273e+00  &   --  &  9.5339e-06  &   --    \\
                                & & 1/{8}  &  5.0961e-03  &  1.75 &  1.1878e-01  &  1.04 &  7.1024e-01  &  1.01 &  2.1197e-06  &  2.17   \\
                                & & 1/{16} &  1.3477e-03  &  1.92 &  5.8490e-02  &  1.02 &  3.4157e-01  &  1.06 &  2.6746e-09  &  9.63  \\
                                & & 1/{32} &  3.4400e-04  &  1.97 &  2.9036e-02  &  1.01 &  1.6712e-01  &  1.03 &  6.7721e-11  &  5.30  \\
                                & & 1/{64} &  8.6762e-05  &  1.99 &  1.4465e-02  &  1.01 &  8.2866e-02  &  1.01 &  3.5851e-12  &  4.24  \\
                                & & 1/{128}&  2.1779e-05  &  1.99 &  7.2184e-03  &  1.00 &  4.1319e-02  &  1.00 &  1.7368e-11  &  -2.28 \\
\cline{2-11}
			& \multirow{6}{*}{2}  & 1/{4}  &  9.2460e-04  &   --  &  3.4138e-02  &   --  &  1.8100e-01  &   --  &  1.3883e-04  &   --    \\
                                & & 1/{8}  &  1.2355e-04  &  2.90 &  8.5794e-03  &  1.99 &  4.6441e-02  &  1.96 &  1.8428e-07  &  9.56   \\
                                & & 1/{16} &  1.6029e-05  &  2.95 &  2.1321e-03  &  2.01 &  1.1725e-02  &  1.99 &  4.1800e-09  &  5.46  \\
                                & & 1/{32} &  2.0410e-06  &  2.97 &  5.3134e-04  &  2.00 &  2.9323e-03  &  2.00 &  3.8990e-11  &  6.74  \\
                                & & 1/{64} &  2.5744e-07  &  2.99 &  1.3266e-04  &  2.00 &  7.3234e-04  &  2.00 &  1.4930e-10  &  -1.94 \\
                                & & 1/{128}&  3.2324e-08  &  2.99 &  3.3146e-05  &  2.00 &  1.8294e-04  &  2.00 &  4.1893e-11  &  1.83  \\
\hline				\end{tabular} 
	\end{center}
	} }
	\caption{Convergence with {\color{black}{IOP-EG}} ($\alpha=1$)}
	\label{table:quadratic-conv}
 \end{table}
\begin{table}[h]
	{\color{black}{
	\begin{center} \tiny 
		\begin{tabular}{c|c|c|cc|cc|cc|cc} \hline
			\multirow{2}{*}{$\kappa_0$} & \multirow{2}{*}{$k$} & \multirow{2}{*}{$h_{\max}$} & \multicolumn{2}{c|}{$ \| u - u_h \|_{0} $} & \multicolumn{2}{c|}{$ \| u - u_h \|_{a_h} $} & \multicolumn{2}{c|}{$ \| \zb - \zb_h \|_{\kap^{-1}}$} & \multicolumn{2}{c}{$ \| P_0(f - \div \zb_h) \|_0 $} \\
			 & & & error & rate & error & rate & error & rate & error & rate \\ \hline \hline 
			\multirow{12}{*}{1} & \multirow{6}{*}{1}  & 1/{4}  &  1.7401e-02  &   --  &  2.3188e-01  &   --  &  1.8320e-01  &   --  &  6.8532e-06  &   --    \\ 
                                & & 1/{8}  &  5.1355e-03  &  1.76 &  1.1747e-01  &  0.98 &  8.2788e-02  &  1.15 &  1.8947e-07  &  5.18   \\ 
                                & & 1/{16} &  1.3566e-03  &  1.92 &  5.8393e-02  &  1.01 &  3.7252e-02  &  1.15 &  7.5461e-10  &  7.97  \\ 
                                & & 1/{32} &  3.4625e-04  &  1.97 &  2.9030e-02  &  1.01 &  1.7675e-02  &  1.08 &  1.6568e-10  &  2.19  \\ 
                                & & 1/{64} &  8.7325e-05  &  1.99 &  1.4464e-02  &  1.01 &  8.6658e-03  &  1.03 &  8.2004e-13  &  7.66  \\ 
                                & & 1/{128}&  2.1919e-05  &  1.99 &  7.2184e-03  &  1.00 &  4.3051e-03  &  1.01 &  3.2178e-12  &  -1.97 \\ 
\cline{2-11} 
			& \multirow{6}{*}{2}  & 1/{4}  &  9.5448e-04  &   --  &  3.2725e-02  &   --  &  2.2288e-02  &   --  &  8.6195e-06  &   --    \\ 
                                & & 1/{8}  &  1.2558e-04  &  2.93 &  8.3760e-03  &  1.97 &  5.7739e-03  &  1.95 &  2.4908e-07  &  5.11   \\ 
                                & & 1/{16} &  1.6143e-05  &  2.96 &  2.1101e-03  &  1.99 &  1.4448e-03  &  2.00 &  7.6211e-09  &  5.03  \\ 
                                & & 1/{32} &  2.0474e-06  &  2.98 &  5.2900e-04  &  2.00 &  3.6016e-04  &  2.00 &  4.2024e-11  &  7.50  \\ 
                                & & 1/{64} &  2.5781e-07  &  2.99 &  1.3240e-04  &  2.00 &  8.9872e-05  &  2.00 &  1.7440e-12  &  4.59  \\ 
                                & & 1/{128}&  3.2346e-08  &  2.99 &  3.3115e-05  &  2.00 &  2.2446e-05  &  2.00 &  6.0409e-12  &  -1.79 \\ 
\hline		
			\multirow{12}{*}{10} & \multirow{6}{*}{1}  & 1/{4}  &  1.7247e-02  &   --  &  2.4384e-01  &   --  &  1.4472e+00  &   --  &  1.0088e-05  &   --   \\
                                & & 1/{8}  &  5.1065e-03  &  1.76 &  1.1880e-01  &  1.04 &  7.1414e-01  &  1.02 &  4.3373e-07  &  4.54  \\
                                & & 1/{16} &  1.3482e-03  &  1.92 &  5.8491e-02  &  1.02 &  3.4198e-01  &  1.06 &  2.0379e-08  &  4.41 \\
                                & & 1/{32} &  3.4403e-04  &  1.97 &  2.9036e-02  &  1.01 &  1.6715e-01  &  1.03 &  1.7823e-10  &  6.84 \\
                                & & 1/{64} &  8.6763e-05  &  1.99 &  1.4465e-02  &  1.01 &  8.2869e-02  &  1.01 &  1.4241e-11  &  3.65 \\
                                & & 1/{128}&  2.1779e-05  &  1.99 &  7.2184e-03  &  1.00 &  4.1319e-02  &  1.00 &  1.2476e-11  &  0.19 \\
\cline{2-11}
			& \multirow{6}{*}{2}  & 1/{4}  &  9.3036e-04  &   --  &  3.4126e-02  &   --  &  1.8096e-01  &   --  &  4.6005e-05  &   --    \\
                                & & 1/{8}  &  1.2421e-04  &  2.90 &  8.5715e-03  &  1.99 &  4.6526e-02  &  1.96 &  4.2893e-07  &  6.74   \\
                                & & 1/{16} &  1.6080e-05  &  2.95 &  2.1306e-03  &  2.01 &  1.1734e-02  &  1.99 &  6.3984e-09  &  6.07  \\
                                & & 1/{32} &  2.0444e-06  &  2.98 &  5.3112e-04  &  2.00 &  2.9330e-03  &  2.00 &  1.2150e-10  &  5.72  \\
                                & & 1/{64} &  2.5766e-07  &  2.99 &  1.3263e-04  &  2.00 &  7.3240e-04  &  2.00 &  4.0834e-11  &  1.57  \\
                                & & 1/{128}&  3.2338e-08  &  2.99 &  3.3142e-05  &  2.00 &  1.8295e-04  &  2.00 &  4.3021e-11  &  -0.08 \\
\hline		
		\end{tabular} 
	\end{center}  
	} }
	\caption{Convergence with {\color{black}{IOP-EG}} ($\alpha = 2$)}
	\label{table:cubic-conv}
\end{table}

{\color{black}{ 
\begin{table}[h] \color{black}{ \tiny
	\begin{center}
		\begin{tabular}{c|c||c c c c c} 
			\multicolumn{2}{c||}{ } & \multicolumn{5}{c}{$N$} \\ 
			$k$ &$\ \kappa_0$ &$8$ &$16$ &$32$ &$64$ &$128$ \\ 
			\hline 
			\multirow{5}{*}{$1$}&$1$ &$27$ &$36$ &$48$ &$64$ &$86$ \\ 
								&$2$ &$29$ &$37$ &$50$ &$67$ &$91$ \\ 
								&$4$ &$29$ &$40$ &$54$ &$73$ &$99$ \\ 
								&$8$ &$30$ &$43$ &$59$ &$81$ &$110$ \\ 
								&$10$&$30$ &$43$ &$59$ &$83$ &$115$ \\ 
			\cline{2-7} 
			\multirow{5}{*}{$2$}&$1$ &$38$ &$45$ &$56$ &$73$ &$95$ \\ 
								&$2$ &$40$ &$47$ &$59$ &$76$ &$101$ \\ 
								&$4$ &$42$ &$51$ &$63$ &$83$ &$110$ \\ 
								&$8$ &$45$ &$55$ &$70$ &$93$ &$124$ \\ 
								&$10$&$46$ &$57$ &$72$ &$97$ &$129$ \\ 
			\hline 
		\end{tabular} 
		\caption{Number of iterations for $\alpha = 0$} 
		\label{table:eg} 
	\end{center} }
\end{table} 
}}

{\color{black}{ 
\begin{table}[h]
	\color{black}{ \tiny
	\begin{center}
		\begin{tabular}{c|c||c c c c c} 
			\multicolumn{2}{c||}{ } & \multicolumn{5}{c}{$N$} \\ 
			$k$ &$\ \kappa_0$ &$8$ &$16$ &$32$ &$64$ &$128$ \\ 
			\hline 
			\multirow{5}{*}{$1$}&$1$ &$17$ &$18$ &$19$ &$19$ &$18$ \\ 
								&$2$ &$18$ &$19$ &$19$ &$19$ &$19$ \\ 
								&$4$ &$19$ &$20$ &$20$ &$19$ &$19$ \\ 
								&$8$ &$20$ &$21$ &$21$ &$21$ &$20$ \\ 
								&$10$&$20$ &$21$ &$21$ &$21$ &$21$ \\ 
			\cline{2-7} 
			\multirow{5}{*}{$2$}&$1$ &$31$ &$32$ &$31$ &$30$ &$28$ \\ 
								&$2$ &$32$ &$33$ &$33$ &$32$ &$30$ \\ 
								&$4$ &$33$ &$34$ &$34$ &$32$ &$32$ \\ 
								&$8$ &$34$ &$35$ &$35$ &$34$ &$33$ \\ 
								&$10$&$35$ &$36$ &$35$ &$34$ &$33$ \\ 
			\hline 
		\end{tabular} 
		\caption{Number of iterations for $\alpha = 1$ 
		} 
		\label{table:quadratic} 
	\end{center} 
	}
\end{table} 
}
}

{\color{black}{ 
\begin{table}[h]
	\color{black}{ \tiny
	\begin{center}
		\begin{tabular}{c|c||c c c c c} 
			\multicolumn{2}{c||}{ } & \multicolumn{5}{c}{$N$} \\ 
			$k$ &$\ \kappa_0$ &$8$ &$16$ &$32$ &$64$ &$128$ \\ 
			\hline 
			\multirow{5}{*}{$1$}&$1$ &$15$ &$15$ &$14$ &$14$ &$13$ \\ 
								&$2$ &$15$ &$15$ &$14$ &$14$ &$14$ \\ 
								&$4$ &$16$ &$15$ &$15$ &$14$ &$14$ \\ 
								&$8$ &$16$ &$15$ &$15$ &$15$ &$14$ \\ 
								&$10$&$16$ &$16$ &$15$ &$15$ &$14$ \\ 
			\cline{2-7} 
			\multirow{5}{*}{$2$}&$1$ &$28$ &$29$ &$29$ &$27$ &$25$ \\ 
								&$2$ &$29$ &$30$ &$29$ &$29$ &$28$ \\ 
								&$4$ &$31$ &$31$ &$30$ &$29$ &$29$ \\ 
								&$8$ &$31$ &$31$ &$31$ &$30$ &$29$ \\ 
								&$10$&$32$ &$31$ &$31$ &$30$ &$29$ \\ 
			\hline
		\end{tabular} 
		\caption{Number of iterations for $\alpha = 2$ 
		} 
		\label{table:cubic} 
	\end{center} 
	}
\end{table} 
}
}

{\color{black}{ 
\begin{table}[h]
	\color{black}{ \tiny
	\begin{center}
		\begin{tabular}{c|c||c c c c c} 
			\multicolumn{2}{c||}{ } & \multicolumn{5}{c}{$N$} \\ 
			$k$ &$\ \kappa_0$ &$8$ &$16$ &$32$ &$64$ &$128$ \\ 
			\hline 
			\multirow{5}{*}{$1$}&$1$ &$15$ &$17$ &$19$ &$22$ &$25$ \\ 
								&$2$ &$16$ &$17$ &$20$ &$23$ &$26$ \\ 
								&$4$ &$16$ &$18$ &$20$ &$23$ &$28$ \\ 
								&$8$ &$17$ &$19$ &$21$ &$25$ &$31$ \\ 
								&$10$&$17$ &$19$ &$22$ &$26$ &$32$ \\ 
			\cline{2-7} 
			\multirow{5}{*}{$2$}&$1$ &$29$ &$31$ &$31$ &$32$ &$35$ \\ 
								&$2$ &$30$ &$32$ &$33$ &$34$ &$37$ \\ 
								&$4$ &$31$ &$33$ &$34$ &$35$ &$38$ \\ 
								&$8$ &$32$ &$34$ &$35$ &$37$ &$41$ \\ 
								&$10$&$32$ &$34$ &$35$ &$37$ &$41$ \\ 
			\hline 
		\end{tabular} 
		\caption{Number of iterations for large $\gamma~(=200)$ on $\mathcal{E}_h^0$
		} 
		\label{table:large} 
	\end{center} 
	}
\end{table} 
}
}

{\color{black}{ 
\begin{table}[h]
	\color{black}{ \tiny
	\begin{center}
		\begin{tabular}{c|c||c c c c c} 
			\multicolumn{2}{c||}{ } & \multicolumn{5}{c}{$N$} \\ 
			$k$ &$\ \kappa_0$ &$8$ &$16$ &$32$ &$64$ &$128$ \\ 
			\hline 
			\multirow{5}{*}{$1$}&$1$ &$21$ &$23$ &$26$ &$29$ &$32$ \\ 
								&$2$ &$22$ &$25$ &$27$ &$31$ &$34$ \\ 
								&$4$ &$23$ &$27$ &$30$ &$33$ &$37$ \\ 
								&$8$ &$24$ &$29$ &$32$ &$36$ &$41$ \\ 
								&$10$&$24$ &$29$ &$34$ &$38$ &$42$ \\ 
			\cline{2-7} 
			\multirow{5}{*}{$2$}&$1$ &$34$ &$35$ &$37$ &$39$ &$41$ \\ 
								&$2$ &$35$ &$37$ &$39$ &$40$ &$42$ \\ 
								&$4$ &$36$ &$39$ &$40$ &$42$ &$45$ \\ 
								&$8$ &$38$ &$41$ &$44$ &$46$ &$48$ \\ 
								&$10$&$39$ &$42$ &$45$ &$47$ &$51$ \\ 
			\hline 
		\end{tabular} 
		\caption{Number of iterations for $\alpha = 0.5$} 
		\label{table:onehalf} 
	\end{center} 
	}
\end{table} 
}
}

\begin{figure}[ht]  
	\hspace{-0.cm}
	\begin{subfigure}[b]{1.05\textwidth}        
		\includegraphics[width=.5\linewidth]{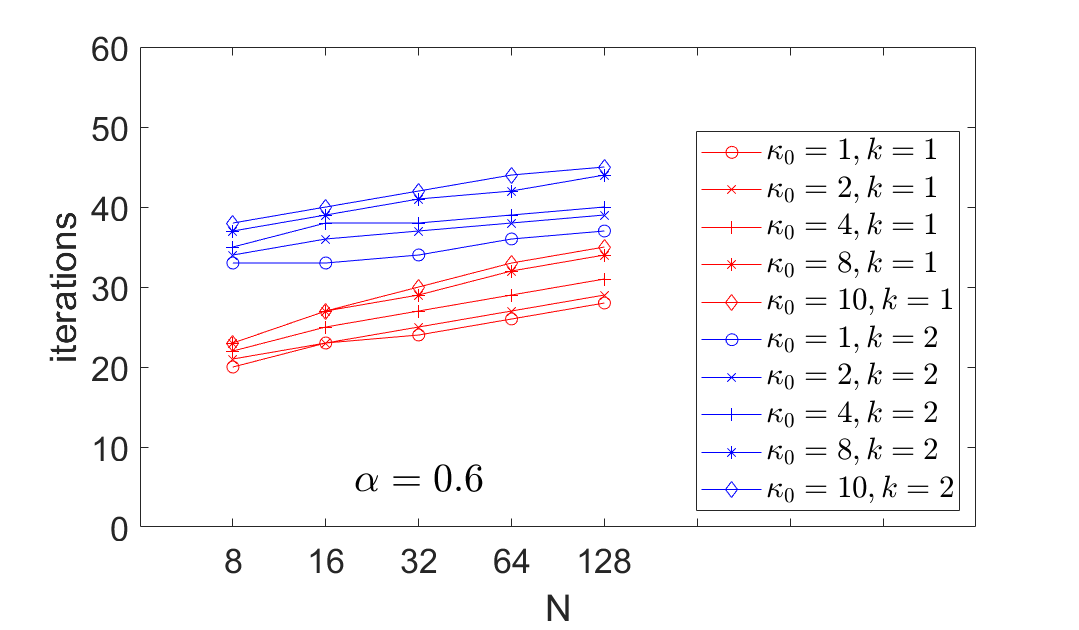} ~\hspace{-5mm}
		\includegraphics[width=.5\linewidth]{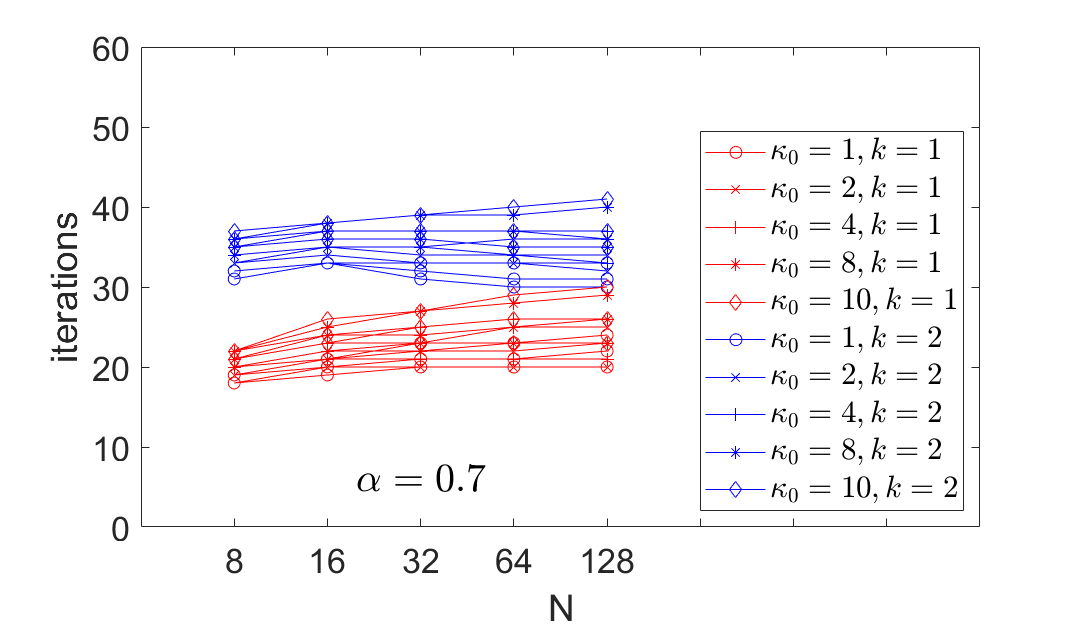} 
    \end{subfigure} 
    \\ 
    \vspace{5mm}
	\hspace{-0.cm}
	\begin{subfigure}[b]{1.05\textwidth}
		\hspace{-2mm}
		\includegraphics[width=.5\linewidth]{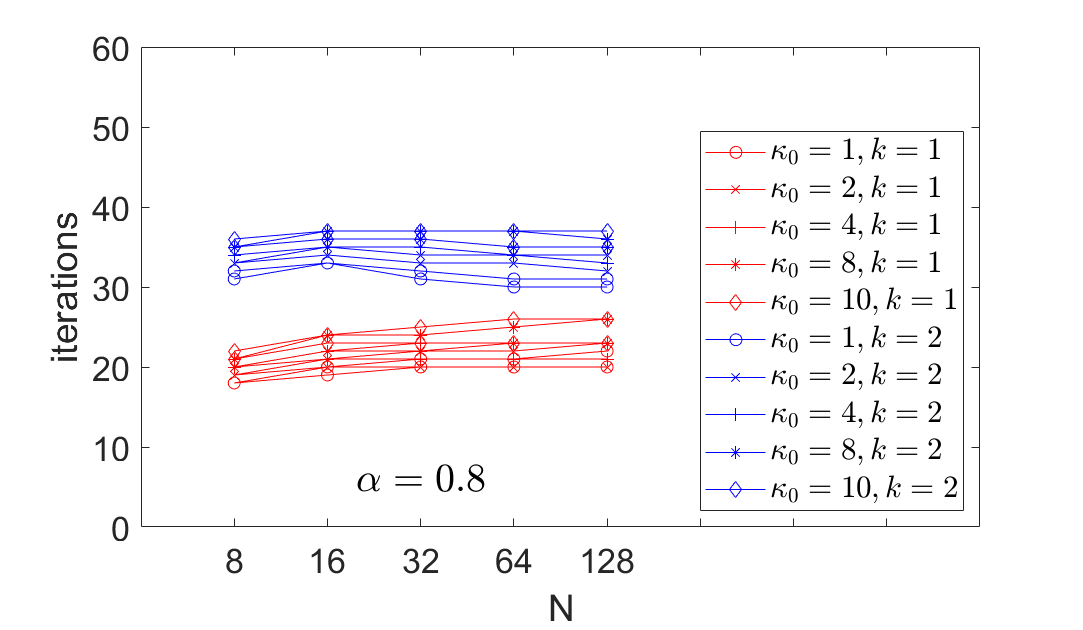} ~\hspace{-5mm}
		\includegraphics[width=.5\linewidth]{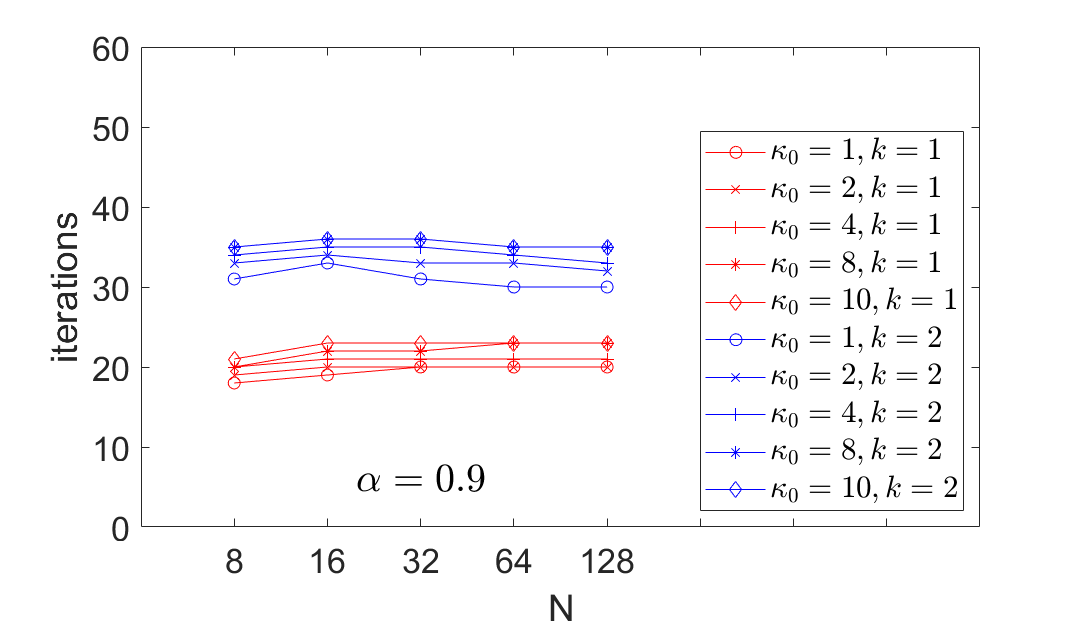} 
  	\end{subfigure}
    \\
	\hspace{-0.cm}
	\begin{subfigure}[b]{1.05\textwidth}
		\hspace{-0mm}
		\includegraphics[width=.5\linewidth]{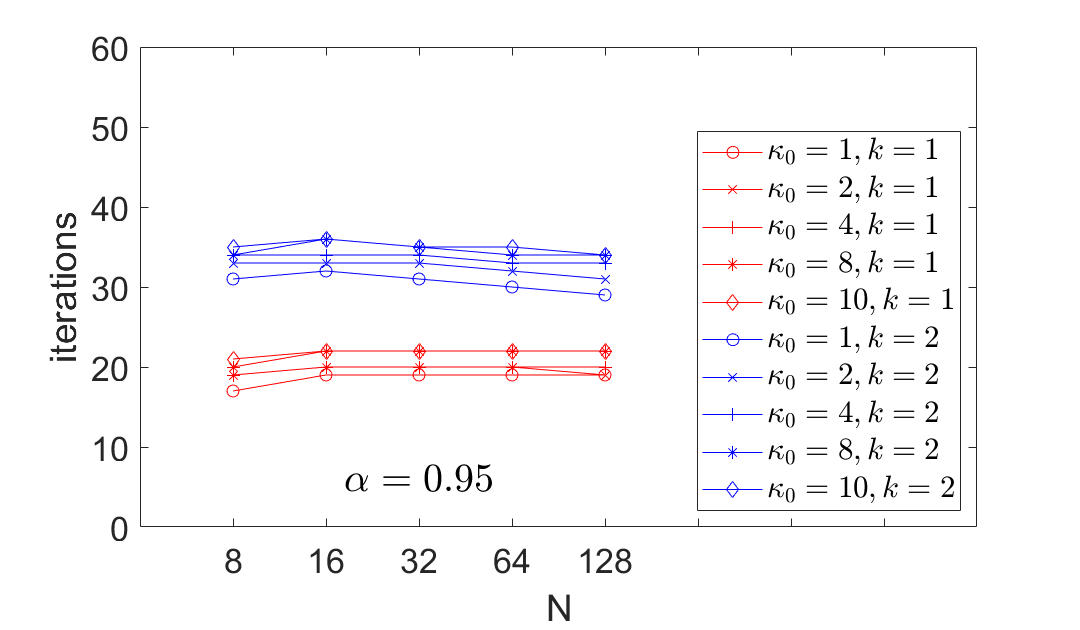} ~\hspace{-5mm}
		\includegraphics[width=.5\linewidth]{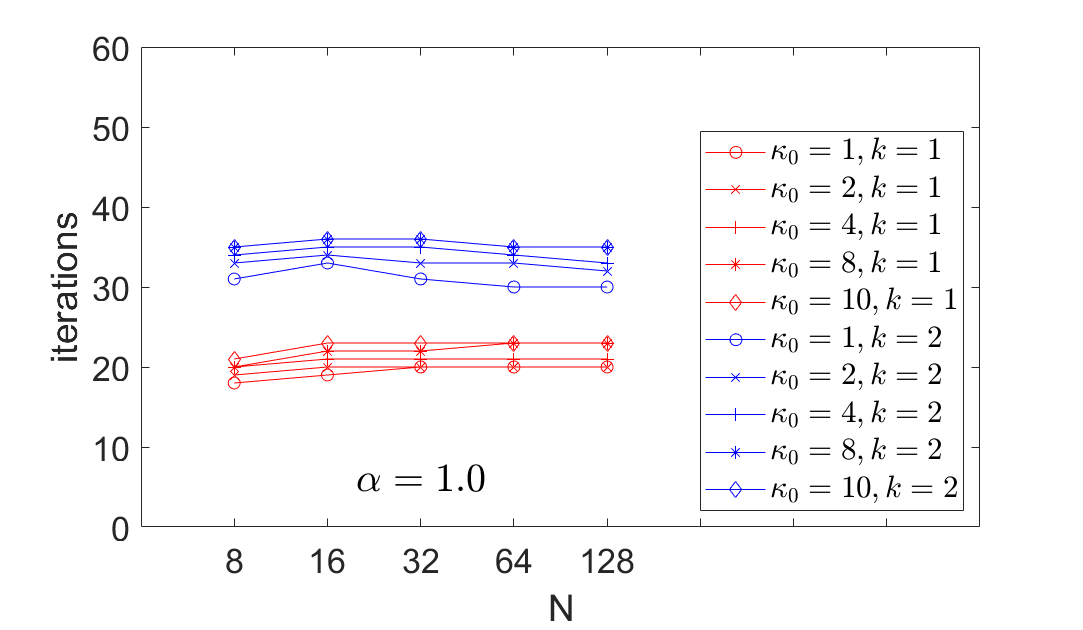} 
  	\end{subfigure}
  	\caption{Comparison of iteration numbers for different $\alpha$ and $\kappa_0$ values}
  	\label{fig:alpha-comparison}
\end{figure}
%

\tred{
In numerical experiments of this section, we consider 
\algn{ \label{eq:aniso-Poisson}
- \nabla \cdot (\kap \nabla u) = f \quad \text{ in } \Omega, \qquad u = 0 \quad \text{ on } \pd \Omega 
}
with a tensor 
\algn{ \label{eq:aniso-kappa}
	\kap = \pmat{\kappa_0 & 0 \\ 0 & 1 } 
}
where $\kappa_0>0$ is a constant. Let $\kap_{n} := \bs{\kap} n \cdot n$ on edge/face $e$ with a unit normal vector $n$ on $e$. Then, a modification of \eqref{eq:ah} for \eqref{eq:aniso-Poisson} is 
\algn{ \label{eq:ah-kappa}
	a_h (v, w) &= \LRp{ \kap \nabla v, \nabla w} - \LRp{ \LRa{ \avg{\kap \nabla v}, \jump{w}}_{\Ehi\cup\EhD} + \LRa{ \jump{v}, \avg{\kap \nabla w}}_{\Ehi\cup\EhD} } \\
	\notag &\quad + \LRa{ \gammakap \kap_{n} h_e^{-1-\alpha} \jump{v}, \jump{w} }_{\Ehi} + \LRa{ \gammakap \kap_{n} h_e^{-1 } \jump{v}, \jump{w} }_{\EhD} ,
}
and flux reconstruction formula is 
{\small
\algns{ 
	(\zb_h, \bs{q})_T &= (- \kap \nabla \uh , \bs{q})_T, &  & \bs{q} \in \Pb_{k-2}(T), \\
	\LRa{ \zb_h \cdot \n, q }_e &= \LRa{- \avg{\kap \nabla \uh} \cdot \n + \gammakap \kap_{n}h_e^{-1-\alpha} \jump{\uh} \cdot \n , q }_e, & & q \in \mc{P}_{k-1}(e), e \in \Ehi , \\
	\LRa{ \zb_h \cdot \n, q }_e &= \LRa{u_N , q }_e, & & q \in \mc{P}_{k-1}(e), e \not\in \Ehi \cup \EhD, \\
\LRa{ \zb_h \cdot \n, q }_e &= \LRa{- {\kap\nabla \uh} \cdot \n + \gammakap \kap_{n} h_e^{-1} (\uh - u_D) , q }_e, & & q \in \mc{P}_{k-1}(e), e \in \EhD .
}
}
For simplicity of presentation we only showed a priori error analysis in Section~\ref{sec:apriori} for $\kappa_0=1$. The error analysis can be extended to \eqref{eq:aniso-Poisson} with $\kap$ weighted flux and penalization terms. In contrast, for the preconditioning discussed in Section~\ref{sec:preconditioning}, a constant which is related to the anisotropy of $\kap$, is involved in the operator preconditioning analysis. Thus, it does not seem to be straightforward to get an analytic proof that abstract preconditioners of the form \eqref{eq:block-diag} are spectrally equivalent to the operator given by \eqref{eq:ah-kappa} with equivalence constants independent of the anisotropy constant. Nonetheless, we present numerical results of convergence and preconditioners for anisotropic coefficients in this section because the numerical test results below show that IOP-EG methods with preconditioner  \eqref{eq:block-diag} are robust for the anisotropy of $\kap$ (see Tables~\ref{table:quadratic}-\ref{table:cubic}). As seen in Table~\ref{table:eg}, preconditioned EG methods show worse performances if $\kap$ is more anisotropic (i.e., $\kappa_0$ is larger), so this anisotropy robustness is another advantage of IOP-EG methods.
}

\tred{
In the first set of numerical experiments we present convergence rates of errors of various versions of {\color{black}{IOP-EG}} methods with $\alpha = 1,2$, $\kappa_0 = 1,10$ in \eqref{eq:ah-kappa} and the manufactured solution $u = x(1-x)\sin(\pi y)$. Recall that a larger $\alpha$ implies a stronger interior over-penalization. From the results in Tables~\ref{table:quadratic-conv}--\ref{table:cubic-conv}, one can see that convergence rates are optimal for the $L^2$ error $\| u - u_h\|_0$, the $a_h$-norm $\| u - u_h \|_{a_h}$ endowed by \eqref{eq:ah-norm} with the bilinear form in \eqref{eq:ah-kappa}, the $\kap^{-1}$-weighted $L^2$ flux error $\|\bs{z} - \bs{z}_h\|_{\kap^{-1}}:=(\kap^{-1} (\bs{z} - \bs{z}_h), \bs{z} - \bs{z}_h)^{1/2}$. The errors of local mass reach the level of machine precision zero quickly as mesh is refined. As shown in the proof of Theorem~\ref{thm:flux-error}, exact local mass conservation theoretically holds for all meshes. However, the error $\| P_0(f - \div \bs{z}_h)\|_0$ is not machine precision zero for coarse meshes in our experiments because the errors of numerical quadrature of manufactured solution $u = x(1-x)\sin(\pi y)$ are involved in flux reconstruction}.

In the second set of experiments we present performance of preconditioners of the form \eqref{eq:block-diag-mat-inv} as a function of mesh refinement and anisotropy of the permeability tensor. 
We set $\kappa_0 = 1, 2, 4, 8, 10$ to test the proposed numerical methods with preconditioners for anisotropic tensors. In the results of experiments, $k$ is the polynomial degree of $\Vhc$, $\kappa_0$ is the coefficient in \eqref{eq:aniso-kappa} and $\gamma =10$. 

Numerical results for iterative solvers with $\alpha = 1,2$ are presented in Tables~\tred{4--5}.
More specifically, we present the number of iterations of a preconditioned MinRes method with the block diagonal preconditioner of the form \eqref{eq:block-diag-mat-inv}, and iteration stops when relative error becomes smaller than $10^{-12}$ of the initial error or \tred{when} the number of iterations is more than $10^4$. 


We use the MinRes method instead of the conjugate gradient method for guaranteed convergence because the matrix is not symmetric positive definite in \tred{either} EG or {\color{black}{IOP-EG}} methods.
We use the hypre library (cf.~\cite{Falgout2002}) as an algebraic multigrid preconditioner for the blocks in \eqref{eq:block-diag-mat-inv}.
The corresponding preconditioner is constructed as in \eqref{eq:block-diag} with $a_h$ in \eqref{eq:ah-kappa}. 


We also tested performance of preconditioners for the original EG method ($\alpha =0$) and \tred{two} {\color{black}{IOP-EG}} methods with $\alpha = 1, 2$. 
The results are given in Tables \ref{table:eg}--\ref{table:cubic}. 
In Table \ref{table:eg}, the number of iterations for the EG method clearly increases for mesh refinement in all cases.
Moreover, the number of iterations increases if $\kap$ is more anisotropic, and the increment is not negligible for the two cases $\kappa_0=1$ and $\kappa_0=10$ in the finest mesh $(N=128)$.
In contrast, the {\color{black}{IOP-EG}} methods with $\alpha = 1, 2$ \tred{perform} much better than the ones of the original EG method in all cases. 
As can be seen in Tables \ref{table:quadratic}--\ref{table:cubic}, the number of iterations is very robust for mesh refinement. 
For $\kap$ anisotropy, the number of iterations increases as $\kap$ becomes more anisotropic. 
Nevertheless, the {\color{black}{IOP-EG}} methods perform much better than the original EG methods in all cases.

Finally, we present two additional preconditioning experiment results. 
The purpose of the first additional experiment is to verify that the mesh-dependent over-penalization is necessary for robust preconditioning for mesh refinement. 
For this we set $\alpha =0$ and use 10 times stronger interior penalization parameter $\gamma$ on the interior edges/faces than the penalization parameter on the boundary edges/faces. 
Specifically, we set $\gamma$ on the interior edges/faces as $\gamma_{\text{int}}=100$ and set $\gamma$ on the boundary edges/faces as $\gamma_{\pd} =10$.
In Table~\ref{table:large} one can see that the number of iterations is smaller than the ones in Table~\ref{table:eg}, so the interior over-penalization with a constant factor clearly improves preconditioning performances. However, the results in Table~\ref{table:large} also show that the number of iterations still considerably increases for mesh refinement. 
The purpose of the second additional experiment is to obtain numerical evidence \tred{of how large $\alpha >0$ need to be for robust preconditioning while our analysis showed that $\alpha \ge 1$ is sufficient if $\kappa_0=1$. For this, we ran preconditioning experiments with different $\alpha$ values and the results are given in Figure~\ref{fig:alpha-comparison}. The results show that the number of iterations increases if $\alpha < 0.9$ while for mesh refinement and anisotropy of $\kap$, the iteration numbers seem to be nearly stable for mesh refinement and anisotropy if $\alpha\ge 0.9$. The numerical results show that while $\alpha = 1$ may not be a sharp threshold value for preconditioning robustness we can observe that $\alpha$ should be sufficiently close to 1. }


\subsubsection*{Acknowledgement}

Omar Ghattas gratefully acknowledge support by Department of Energy, Office of Advanced Scientific Computing Research (award number DE-SC0019303).
Jeonghun J. Lee is partially supported by University Research Committee grant of Baylor University and by the National Science Foundation under grant number DMS-2110781.

\section{Conclusion}
\label{sec:conclusion}
\tred{
In this paper we propose interior over-penalized enriched Galerkin (IOP-EG) methods for second order elliptic equations. 
From theoretical point of view, a new medius error analysis and a new spectral equivalence analysis, are developed for optimal convergence of errors and for construction of robust iterative solvers. 
In numerical experiment results comparing preconditioned IOP-EG and EG methods, we observe that preconditioned IOP-EG methods show very robust iterative solver performances for mesh refinement and for the anisotropy of permeability coefficients. In conclusion, the IOP-EG methods can be a good replacement of the original EG methods providing parameter-robust scalable iterative solvers.
}




\bibliographystyle{elsarticle-num}

\end{document}